\documentclass[12pt,leqno]{amsart}
\usepackage{times}
\usepackage{amsmath,amssymb,amscd,amsthm}
\usepackage{latexsym}
\usepackage{epsfig}

\newtheorem{theorem}{Theorem}
\newtheorem{lemma}{Lemma}
\newtheorem{proposition}{Proposition}
\newtheorem{remark}{Remark}

\newtheorem{corollary}{Corollary}

\newcommand{\na}{\nabla}
\newcommand{\p}{\partial}

\newcommand{\n}[2]{{\left\| #1 \right\|}_{#2}}
\newcommand{\f}[2]{\frac{#1}{#2}}

\newcommand{\lan}[1]{\left\langle #1\right\rangle}
\newcommand{\wh}[1]{\widehat{#1}}
\newcommand{\wt}[1]{\widetilde{#1}}

\newcommand{\al}{\alpha}
\newcommand{\be}{\beta}
\newcommand{\ga}{\gamma}

\newcommand{\de}{\delta}

\newcommand{\ve}{\varepsilon}

\newcommand{\La}{\Lambda}
\newcommand{\si}{\sigma}

\newcommand{\bt}{{\mathbf T}}
\newcommand{\br}{{\mathbf R}}

\newcommand{\cs}{\mathcal S}

\newcommand{\cf}{\mathcal F}

\begin{document}

\title
[Resonant phase-shift of periodic KdV evolution]
{Resonant phase-shift and global smoothing of the periodic Korteweg-de Vries equation in low regularity settings}

\author{Seungly Oh}

\address{Seungly Oh, 
405, Snow Hall, 1460 Jayhawk Blvd. , 
Department of Mathematics,
University of Kansas,
Lawrence, KS~66045, USA}
\date{\today}



\begin{abstract}
We show a smoothing effect of near full derivative for low-regularity global-in-time solutions of the periodic Korteweg-de Vries (KdV) equation.  The smoothing is given by slightly shifting the space-time Fourier support of the nonlinear solution, which we call \emph{resonant phase-shift}.  More precisely, we show that $[\mathcal{S}u](t) - e^{-t\p_x^3} u(0) \in H^{-s+1-}$ where $u(0) \in H^{-s}$ for $0\leq s <1/2$ where $\mathcal{S}$ is the resonant phase-shift operator described below.  We use the normal form method to obtain the result.
\end{abstract}

\maketitle
\date{today}

\section{Introduction}
Consider the real-valued periodic Korteweg-de Vries (KdV) equation
\begin{equation}\label{kdvt}
\left|\begin{array}{l} u_t + u_{xxx} = \partial_x ( u^2); \qquad (t,x)\in \mathbf{R}\times \mathbf{T}\\
u(0,x)= u_0 \in H^{-s} (\mathbf{T})\end{array}\right. 
\end{equation}
where $\mathbf{T} = \mathbf{R} \mod [0,2\pi]$.  This initial value problem has been extensively studied in literature. In \cite{Bour}, Bourgain developed a weighted Sobolev space to prove the global well-posedness of \eqref{kdvt} on $L^2$.  The key idea of Bourgain space is to penalize space-time Fourier support away from that the of the linear solution~$\{(\tau,\xi): \tau-\xi^3 =0\}$, thus assuming that the nonlinear solution is near the free solution~$e^{-t\p_x^3} u_0$.  

Kenig, Ponce, Vega \cite{KPV1} adapted this idea to prove the local well-posedness of \eqref{kdvt} in $H^s$ for $s>-1/2$; and Colliander, Keel, Staffilani, Takaoka and Tao \cite{I1} improved the local theory to $s\geq -1/2$ and also proved the global well-posedness of \eqref{kdvt} for $s\geq -1/2$.  Here, the authors introduced \emph{I-method} for constructing almost-conserved quantities in order to iterate the local solutions to an arbitrary time interval~$[0,T]$.
 
In \cite{CCT}, Christ, Colliander, Tao proved that \eqref{kdvt} is ill-posed in $H^{s}$ for $s<-1/2$ in the sense that the solution map fails to be uniformly continuous.  In particular, this implies that any attempt at the contraction argument around $e^{-t\p_x^3}u_0$ will fail.  However, Kappeler and Topalov \cite{Kap} showed that \eqref{kdvt} is globally well-posed in $H^{-1}$ via inverse scattering method. This result implies that although the solution map of the periodic KdV is not smooth below $H^{-1/2}$, it is $C^0$ up to $H^{-1}$.  Molinet, \cite{mol1, mol2} showed that this result is sharp in the sense that the solution map of \eqref{kdvt} is discontinuous at every $C^{\infty}$~point in $H^{-1-}$.

For initial data in $L^2$, nearly linear dynamics of the nonlinear evolution of the periodic KdV has been studied by Babin, Ilyn, Titi \cite{titi} and by Erdogan, Tzirakis, Zharnitsky \cite{niko2, niko3}.  Furthermore, Erdogan, Tzirakis \cite{niko} proved that, given $u_0\in H^{s}$ for $s>-1/2$, the solution~$u(t)$ of \eqref{kdvt} satisfies $u(t) - e^{-t\p_x^3} u_0 \in H^{s_1}$ where $s_1 <\min(3s+1, s+1)$.  

The last result by Erdogan and Tzirakis is the standard notion of nonlinear smoothing, based on the idea that the nonlinear solution is a smooth perturbation of the linear one.  Bourgain in \cite{bour3, bour4} noted that the nonlinear Duhamel term is in many cases smoother than the initial data.  This proved to be a useful tool for analysing the growth bound of high Sobolev norms in dispersive PDEs.  Colliander, Staffilani and Takaoka \cite{Col} also used this smoothing effect to show the global well-posedness of KdV on $\mathbf{R}$ below $L^2$.  We remark that this was also the main heuristic behind the author's works with Stefanov \cite{OS2} and \cite{OS1} , where we proved the local well-posedness respectively  of the periodic ``good'' Boussinesq equation in $H^{-3/8+}$ and of 1-D quadratic Schr\"odinger equation in $H^{-1+}$.  The smoothing of such type was necessary in these settings because the standard bilinear Bourgain space estimates were shown to fail below $H^{-1/4}$ \cite{FS}  and $H^{-3/4}$ \cite{KPV2} respectively.

Due to the non-uniform continuity statement of \cite{CCT}, it is clear that the linear solution no longer dominates the evolution when $s<-1/2$.  As expected, the smoothing of such type given in \cite{niko2} vanishes as $s \searrow -1/2$.  In this paper, we will show that when the nonlinear solution is considered to be a perturbation of the resonant evolution term~$R^*[u_0]$ described below, one still gains almost a full derivative in $H^{-\f{1}{2}+}$. That is, when $u_0 \in H^s$, $u(t) -R^*[u_0](t) \in H^{s+1-}$ for all $s>-1/2$.  This can also be expressed as $[\mathcal{S}u](t) - e^{-t\p_x^3} u_0 \in H^{s+1-}$, where $\mathcal{S}$ is a continuous phase-shift operator in $H^s$ described in \eqref{def:phase}.  We will also show that when $u_0$ lives near $L^2$, the effect of the phase-shift is in fact smoothing, i.e. $[\mathcal{S}v](t) - v(t) \in H^{1+3s}$ for any $v(t) \in H^s$ for $s>-1/2$.  Thus, the fully nonlinear smoothing effect is also recovered by triangular inequality.

This non-resonant smoothing effect can be regarded as an evidence that the solutions of \eqref{kdvt} tends more toward $R^*[u_0]$, which is essentially the linear solution with a \emph{resonant phase-shift} in low-regularity settings.  This is more convincing when considering that $R^*[\cdot]$ is not uniformly continuous for $s<-1/2$ (see Remark~\ref{rem:nonuniform}).  Thus in the $C^0$~evolution of KdV in $H^s$ for $-1\leq s\leq -1/2$, it appears that $R^*[u_0]$ should dominate the nonlinear evolution.

Another evidence of this phenomena is found in the analysis of periodic modified KdV equation, i.e. \eqref{kdvt} with $\p_x(u^2)$ replaced with $\p_x(u^3)$.  The failure of uniform continuity is shown \cite{CCT} in $H^s$ for $s<1/2$, although $C^0$~global well-posedness in $L^2$ is implied \cite{KT} the periodic KdV via the bi-analyticity of Miura map.  Although harmonic analysis methods have not been able to reproduce the well-posedness at $L^2$, Takaoka, Tsutsumi \cite{TT} and Nakanishi, Takaoka, Tsutsumi \cite{NTT2} showed the local well-posedness of mKdV in $H^{3/8}$ and in $H^{1/3}$ respectively by considering the resonant phase-shift corresponding to the periodic mKdV.  A new type of Bourgain space was developed here, which penalizes functions whose space-time Fourier support is far from that of the resonant solution.  Considering the uniform discontinuity of the resonant solution below $H^{1/2}$ \cite[Exercise 4.21]{tao2} similar to Remark~\ref{rem:nonuniform}, their results seem to indicate that the resonant solution dominates the evolution of the periodic mKdV below $H^{1/2}$.
  
The main technique used in this paper is the \emph{normal form method}.  This method was first introduced by Shatah \cite{shatah1} in the study of Klein-Gordon equation with a quadratic non-linearity.  It has been adapted for many different types of quasi-linear disersive PDEs, for instance \cite{titi,NTT2,TT} etc.  Recently this concept was reformulated Germain, Masmoudi, Shatah, \cite{GMS1,GMS2} as the \emph{space-time resonance method}, and the authors produced a number of new results in literature using this method.

For the periodic KdV, Babin, Ilyin, Titi \cite{titi} applied the normal form transform (or \emph{differentiation by parts}), which smoothed the non-linearity as observed by Erdogan, Tzirakis in \cite{niko}.  The authors noted that the transform results in a trilinear resonant term, which can be reduced to an \emph{almost-linear} term due to massive cancellations as in the periodic mKdV \cite{Bour, TT}.  This is essentially why the resonant correction in this model is a simple phase-shift, rather than a \emph{genuinely non-linear} operation (i.e.  there are no nonlinear interactions among different space-time Fourier modes).  Instead of adapting the \emph{differentiation by parts} approach of \cite{titi}, we use a bilinear (and a trilinear) pseudo-differential operator to perform the normal form transform.  These operators are expressed as $T(\cdot, \cdot)$ and $J(\cdot,\cdot,\cdot)$ in Section~\ref{sec:setup}.

The following is the main result of this work:

\begin{theorem}\label{thm2}
Let $0\leq s<1/2$, $0<10\de <1-2s$, $0< \ga \leq 1-10\de$.  For any real-valued $u_0 \in H^{-s}(\mathbf{T})$ with $\widehat{u_0}(0) =0$, the solution~$u(t)$ of \eqref{kdvt} satisfies $u(t) - R^*[u_0](t) \in  H^{-s+\ga}(\mathbf{T})$ for all $t\in\mathbf{R}$.  More precisely, there exist constants $C= C(\de, \|u_0\|_{H^{-s}})$ and $\al(\de)=O(1/\de)$ such that
\begin{equation}\label{eq:smoothing}
\n{\mathcal{S}[u](t) - e^{-t\p_x^3} u_0}{H^{-s+\ga}} \leq C \lan{t}^{\al(\de)}
\end{equation}
where we define the \emph{resonant phase-shift operator}~$\mathcal{S}$ to be 
\begin{equation}\label{def:phase}
\cf_x [\mathcal{S}u](t,\xi) := \exp \left(-2i\f{|\wh{u_0}|^2(\xi)}{\xi}t\right) \wh{u}(t,\xi).
\end{equation}

Furthermore, we can write the Lipschitz property of the solution map in a smoother space:
\begin{equation}\label{eq:thmlip}
\|u(t)-v(t)\|_{H_x^{-s+\ga}} \leq C_{N,\delta} \lan{t}^{\al(\de)}\|u_0 - v_0\|_{H_x^{-s +\ga} },
\end{equation}
where $\|u_0\|_{H^{-s}} + \|v_0\|_{H^{-s}} <N$.
\end{theorem}

It should be noted that although one can take any $\ga<1$ above, the constant~$C$ and the power~$\al(\de)$ blow up as $\de \searrow 0$.  

The effect of $\mathcal{S}$ is a shift of space-time Fourier transform of $u$ by $2|\wh{u_0}|^2(\xi)|/ \xi$.  Note that when $u_0 \in H^{-s}$ for $s<-1/2$ and $\wh{u_0}(0)= 0$, such phase-shift is continuous in because $|\wh{u_0}(\xi)|^2/ \xi$ is a bounded quantity.  This is also shown in Corollary~\ref{trivial}.

If we define $R^*[u_0]$ according to the definition given in \eqref{defr} and statements following, \eqref{eq:smoothing} is equivalent to
\begin{equation}\label{eq:smoothing2}
\n{u(t) - R^*[u_0](t)}{H^{-s+\ga}} \leq C\lan{t}^{\al(\de)}.
\end{equation}
Thus, we will prove \eqref{eq:smoothing2} in place of \eqref{eq:smoothing}.

We derive a fully nonlinear smoothing result without a phase-shift (cf. \cite{niko}).

\begin{corollary}\label{trivial}
Let $0\leq s < 1/2$.  Then the solution $u$ of \eqref{kdvt} with the initial data $u_0\in H^{-s}$ satisfies the following non-linear smoothing property $u(t) - e^{-t\p_x^3} u_0 \in H^{-s+\si}$ for $0\leq \si < 1-2s$.  More precisely, the same $C$, $\al(\de)$ in Theorem~\ref{thm2} satisfies
\[
\n{u(t) - e^{-t\p_x^3} u_0 }{H^{-s+\si}} \leq C \lan{t}^{\al(\de)} 
\]
\end{corollary}

\begin{proof}
By triangular inequality,
\[
\n{u(t) - e^{-t\p_x^3} u_0}{H^{-s+\si}} \leq \n{u(t) - [\mathcal{S}u](t)}{H^{-s+\si}} + \n{[\mathcal{S}u](t)- u(t)}{H^{-s+\si}}
\]
where $\mathcal{S}$ is defined in \eqref{def:phase}.  Note that the first term on RHS belongs in $C^0_t H_x^{-s+\si}$ for $\si < 1$ by \eqref{eq:smoothing}. The second term is estimated via the following claim.

\textbf{Claim:} Let $u_0 \in H^{-s}$ with mean zero.  Then if $v(t)\in H^{\al}$ for any $t, \al\in \br$, 
\[
\n{\mathcal{S}[v](t) - v(t)}{H_x^{\al + 1-2s}} \lesssim |t|\|u_0\|_{H^{-s}}^2 \|v(t)\|_{H^{\al}}.
\]

To prove the claim, we use mean-value theorem.
\begin{align*}
\n{\mathcal{S}[v](t) - v(t)}{H^{\al + 1-2s}_x} &= \n{\lan{\xi}^{\al + 1-2s}\wh{v}(t,\xi) \left(\exp \left(-2i\frac{|\widehat{u_0}(\xi)|^2}{\xi}  t\right) -1 \right)}{l^2_{\xi}}\\
	&\lesssim \n{\lan{\xi}^{\al+ 1-2s} |\wh{v}|(t,\xi) |t|\frac{|\widehat{u_0}(\xi)|^2}{\xi}}{l^2_{\xi}(\mathbb{Z}\setminus\{0\})}\\
	&\lesssim |t|\|\lan{\xi}^{-2s}|\wh{u_0}|^2\|_{l^{\infty}_{\xi}} \|\lan{\xi}^{\al} |\wh{v}|(t,\xi)\|_{l^{2}_{\xi}} \leq |t|\|u_0\|^2_{H^{-s}} \|v(t)\|_{H^{\al}_x}.
\end{align*}
\end{proof}

The paper is organized as follows. In Section \ref{s2}, we introduce the $X^{s,b}$ spaces and discuss previously obtained results on \eqref{kdvt}.  Section~\ref{s3} contains the proof of Theorem~\ref{thm2} in the following manner: In Section~\ref{sec:setup}, we construct the normal form and perform a few change of variables to reach the new formulation of the equation \eqref{kdvt} to optimize the smoothing effect.  In \ref{sec:properties}, we derive mapping properties of the resonant solution operator~$R[\cdot]$ and the normal form operators~$T(\cdot, \cdot)$ and $J(\cdot, \cdot, \cdot)$.  Section~\ref{sec:nonlinearity} contains the main  multilinear estimates necessary for the contraction mapping, and Section~\ref{sec:local} contains the proof of the local statement for Theorem~\ref{thm2}. In Section \ref{global}, we conclude the proof of the theorem by iterating the local steps and using the global-in-time bound obtained in \cite{I1}.\\

{\bf Acknowledgement:} The author thanks Atanas Stefanov and Vladimir Georgiev for helpful suggestions, and Nikolaos Tzirakis for pointing out some key references.

\section{Notations and preliminaries}\label{s2}
\subsection{Notations}\label{s21}
We adopt the standard notations in approximate inequalities as follows:
By $A \lesssim B$, we mean that there exists an absolute constant $C>0$ with $A \leq B$.  $A \ll B$ means that the implicit constant is taken to be a \emph{sufficiently} large positive number.  For any number of quantities $\alpha_1, \ldots, \alpha_k$, $A\lesssim_{\alpha_1, \ldots, \alpha_k} B$ means that the implicit constant depends only on $\alpha_1, \ldots, \alpha_k$.
Finally, by $A\sim B$, we mean $A\lesssim B$ and $B\lesssim A$.

We indicate by $\eta$ a smooth time cut-off function which is supported on $[-2,2]$ and equals $1$ on $[-1,1]$.  Notations here will be relaxed, since the exact expression of $\eta$ will not influence the outcome.  For any normed space~$\mathcal{Y}$, we denote the quantity $\|\cdot\|_{\mathcal{Y}_T}$ by the expression $\|u\|_{\mathcal{Y}_T} = \|\eta(t/T)u \|_{\mathcal{Y}}$.\\

The spatial, space-time Fourier transforms and spatial inverse Fourier transform are
\begin{align*}
 \widehat{f}(\xi) &= \int_{\mathbf{T}} f(x) e^{-i x\xi}\, dx, \\
 \widetilde{u}(\tau,\xi) &= \int_{ \mathbf{T}\times\mathbf{R}} u(t,x) e^{-i(x\xi+ t\tau)} \,dx\, dt\\
 \mathcal{F}_{\xi}^{-1}[a_{\xi}](x) &= \frac{1}{2\pi} \sum_{\xi \in \mathbf{Z}} a_{\xi} e^{i\xi x}
\end{align*}
where $\xi \in \mathbf{Z}$.  
If $u$ is real-valued, then by above definition $\widehat{u}(-\xi) = \overline{\widehat{u}} (\xi)$.  This is in fact an essential ingredient in our proof, which makes the resonant term autonomous in time.

For a reasonable expression~$\sigma$, we denote differential operators with symbol $\sigma (\cdot)$ via $\sigma (\nabla) f = \sigma(\p_x) f := \mathcal{F}_{\xi}^{-1} [ \sigma(i\xi) \widehat{f}(\xi)]$.  Also, we define $\langle \xi \rangle := 1+|\xi|$.\\ 

\subsection{$X^{s,b}$ spaces and  and local-wellposedness theory.}
 \label{s22}
Bourgain spaces are constructed as the completion of \emph{smooth} functions with with respect to the norm
$$
\|u\|_{X^{s,b}} := \left( \sum_{\xi\in \mathbf{Z}} (1+ |\xi|)^{2s} (1+ |\tau-\xi^3|)^{2b} |\widetilde{u}(\tau,\xi)|^2 \, d\tau \right)^{\frac{1}{2}}.
$$

 The expression above shows $X^{s,0} = L^2_t H^s_x$.  The added weight~$\tau-\xi^3$ (called \emph{modulation frequency}) measures the distance between a space-time Fourier support of the function and the Fourier support of the linear solution~$e^{-t\p_x^3}u_0$.  For instance, the free Airy solution with $L^2$~initial data lies in this space, given an appropriate time cut-off~$\eta \in \mathcal{S}_t$.
\begin{equation}\label{airyfree}
\| \eta (t) e^{-t\p_x^3}f\|_{X^{0,b}} = \|(1+| \tau-\xi^3|)^{b} \widehat{\eta}(\tau-\xi^3) \widehat{f}(\xi)\|_{L^2_{\tau} l^2_{\xi}} \lesssim_{\eta, b} \|f\|_{L^2_x}
\end{equation} 
due to the fast decay of $\widehat{\eta} \in \mathcal{S}_t$.  For $\varepsilon, \delta>0$, we have the embedding properties obtained in Bourgain \cite{Bour},
\begin{align}
\|u\|_{C^0_t H^s_x (\mathbf{R}\times \mathbf{T})} &\lesssim_{\delta} \|u\|_{X^{s,\frac{1}{2}+\delta}},\label{xsb1}\\
\|u\|_{L^4_{t,x}([0,1]\times \mathbf{T})} &\lesssim_{\varepsilon, \delta} \|u\|_{X^{\varepsilon,\frac{1}{3}}},\label{xsb3}\\
\|u\|_{L^6_{t,x}([0,1]\times \mathbf{T})} &\lesssim_{\varepsilon, \delta} \|u\|_{X^{\varepsilon,\frac{1}{2}+\delta}}.\label{xsb2}
\end{align}

The following two estimates provide a convenient framework in our proof.  The proofs in \cite[Prop. 2.12, Lemma 2.11]{tao2}, which argues for $x\in \mathbf{R}^d$  are also valid for the periodic case.
\begin{proposition}\label{xsb}
For $\de>0$ and $s\in \br$,
\[
\|\eta (t) \int_{0}^t e^{-(t-s)\p_x^3} F(s) \, ds\|_{X^{s,\frac{1}{2}+\delta}} \lesssim_{\eta,\delta} \|F\|_{X^{s,-\frac{1}{2}+\delta}}.
\]
\end{proposition}

\begin{proposition} \label{timeloc}
Let $\eta \in \mathcal{S}_t (\mathbf{R})$ and $T\in (0,1)$.  Then for $-\frac{1}{2}< b' \leq b <\frac{1}{2}$, $s\in \mathbf{R}$,
\[
\|\eta(t/T) u \|_{X^{s,b'}} \lesssim_{\eta, b,b'} T^{b-b'} \|u\|_{X^{s,b}}.
\]
\end{proposition}

We use Proposition~\ref{timeloc} to gain positive powers of $T$ by yielding a small portion of modulation weight~$\lan{\tau-\xi^3}$.  

Mean-zero assumption of the initial data is standard through literature for the periodic KdV equation.  This is justified by the mean conservation property of \eqref{kdvt}.  By a change of variable $v(t)= u(t)-\int_{\bt}u_0$ in \cite{Bour}, one can transfer the mean-zero condition to a order 1 perturbation, which can then be removed by gauge transform $t':=t-cx$.  Thus we assume that all initial data~$u_0$ satisfies the mean-zero condition.

We define a closed subspace $Y^{s,b}$ of $X^{s,b}$ (with the same norm) as the image of orthogonal projection $\mathbf{P}: X^{s,b} \to Y^{s,b}$ defined by $\displaystyle \mathbf{P} (u) (x) := u(x) - \int_{\mathbf{T}} u\, dx$.  This will enforce the mean-zero condition on the contraction space.

The analysis of periodic KdV has developed around $X^{-s,\f{1}{2}}$ space, rather than $X^{s,\f{1}{2}+}$ which is properly contained in $C^0_t H^s_x$.  This is mainly due to the failure of $X^{s,b}$ bilinear estimate for $b>1/2$, \cite{KPV1}.  This was remedied in \cite{I1} by introducing $Y^{-s,\f{1}{2}}\cap \wt{l^2_{\xi}(\lan{\xi}^{-s}; L^1_{\tau})}$ where
\[
\|u\|_{\wt{l^2_{\xi}(\lan{\xi}^{-s}; L^1_{\tau})}} := \n{\lan{\xi}^{-s} \wt{u}(\tau,\xi)}{l^2_{\xi} L^1_{\tau}}.
\]
Clearly, $Y^{-s,\f{1}{2}+} \subset \wt{l^2_{\xi}(\lan{\xi}^{-s}; L^1_{\tau})} \subset C^0_t H^{-s}$.  Use of the space $Y^{-s,\f{1}{2}}\cap \wt{l^2_{\xi}(\lan{\xi}^{-s}; L^1_{\tau})}$ increases the number of estimates required for contraction, thus one would prefer to show contraction in $Y^{-s,\f{1}{2}+}$.  In this paper, we do not encounter the same problem described in \cite{KPV1}, since our estimates will be essentially trilinear.  However, we will take advantage of the fact that the global-in-time solution lives in $Y^{-s,\f{1}{2}}\cap \wt{l^2_{\xi}(\lan{\xi}^{-s}; L^1_{\tau})}$.  The following statement is obtained from \cite{I1}.

\begin{proposition}\label{pro:timebound}
The PBIVP is globally well-posed in $H^{-s}$ when $s\in (0,1/2]$.  Furthermore, the solution~$\eta(t/T)u \in Y^{-s,\f{1}{2}}\cap \wt{l^2_{\xi}(\lan{\xi}^{-s}; L^1_{\tau})}$ for any arbitrary time $T$ where $\eta$ is a smooth cut-off function, and for any $T>0$,
\[
\|u(T)\|_{H^{-s}} \lesssim 
\|\eta(\cdot/T) u(\cdot)\|_{\wt{l^2_{\xi} (\lan{\xi}^{-s}; L^1_{\tau})}} \lesssim_{\eta} \lan{T}^s \|u_0\|_{H^{-s}}.
\]
\end{proposition}

\section{Proof of Theorem \ref{thm2}}
\label{s3}

\subsection{Setting of the problem}
\label{sec:setup}
We now turn our attention to \eqref{kdvt}.  Let $u$ be the global-in-time solution of \eqref{kdvt} described in Proposition~\ref{pro:timebound}.  Setting $v := \langle \nabla \rangle^{-s} u$, $v$ satisfies 
\begin{equation}\label{periodic}
v_t + v_{xxx} = \mathcal{N} (v,v), \qquad v(0) = f := \lan{\na}^{-s} u_0 \in L^2(\mathbf{T})
\end{equation}
where $\mathcal{N} (u,v) := \p_x \langle \nabla \rangle^{-s} [\langle \nabla \rangle^s u \langle \nabla \rangle^s v ]$.  In particular, the bilinear operator~$\mathcal{N}$ contains a spatial derivative, thus $\mathcal{N} \equiv \mathbf{P} \circ \mathcal{N}$ itself has mean-zero.

We construct the bilinear pseudo-differential operator $T$ by the formula
$$
T (u,v) := -\f{1}{3}\sum_{\xi_1 \xi_2 (\xi_1+\xi_2) \neq 0} \frac{\langle\xi_1\rangle^s \langle\xi_2\rangle^s}{\langle\xi_1 + \xi_2\rangle^s} \frac{1}{\xi_1 \xi_2} \widehat{u}(\xi_1) \widehat{v} (\xi_2) e^{i(\xi_1 + \xi_2) x}.
$$

Considering the algebraic identity
\[
(\tau_1 +\tau_2) - (\xi_1 + \xi_2)^3 = (\tau_1 - \xi_1^3) + (\tau_2 - \xi_2^3) -3\xi_1 \xi_1(\xi_1 + \xi),
\]
the differential operator $\p_t - \p_{xxx}$ acts on $T$ in the following manner:

$$
(\p_t + \p_{xxx}) T (u,v) = T ((\p_t + \p_{xxx}) u,v) + T(u,(\p_t+ \p_{xxx}) v) + \mathcal{N} (\mathbf{P}u,\mathbf{P}v).
$$
If we write $h= T (v,v)$ where $v$ solves \eqref{periodic} (recall $v=\mathbf{P}v$) and change variable by $v= h+z$, then $z$ satisfies
\begin{equation}\label{zperiodic}
\left| \begin{array}{l} (\p_t + \p_{xxx}) z = -2 T (\mathcal{N}(v,v), v);\\
	z(0) = f-T(f,f). \end{array}\right.
\end{equation}

For the right side of \eqref{zperiodic}, we note that 
$$
T(\mathcal{N}(v,v),v) = \mathbf{P}\langle \nabla \rangle^{-s} (\mathbf{P}[\langle \nabla \rangle^s v \langle \nabla \rangle^s v] \frac{\langle \nabla \rangle^{s}}{\nabla} v).
$$
  
We adapt the computations in \cite{NTT2} to simplify Fourier coefficients of the above expression as follows.  For $\xi\neq 0$ (recall $\widehat{v}(0)=0$ and $\widehat{v}(-\xi) = \overline{\widehat{v}}(\xi)$),

\begin{align*} 
\mathcal{F}[\mathbf{P}\langle \nabla \rangle^{-s} (\mathbf{P}[\langle \nabla \rangle^s v \langle \nabla \rangle^s v] \frac{\langle \nabla \rangle^{s}}{\nabla} v)](\xi) &= \sum_{\tiny \begin{array}{c}\xi_1+\xi_2\neq 0, \quad\xi_j \neq 0\\ \xi_1+\xi_2+\xi_3 = \xi\end{array}} \frac{\langle \xi_1\rangle^s \langle\xi_2\rangle^s \langle \xi_3\rangle^{s}}{i\xi_3\langle \xi\rangle^s} \widehat{v}(\xi_1) \widehat{v}(\xi_2) \widehat{v}(\xi_3)\\
	&\hspace{-150pt}= \sum_{\tiny \begin{array}{c}(\xi_1+\xi_2)(\xi_2+\xi_3)(\xi_3+\xi_1)\neq 0\\ \xi_1+\xi_2+\xi_3=\xi, \quad\xi_j \neq 0 \end{array} } \frac{\langle \xi_1\rangle^s \langle\xi_2\rangle^s \langle \xi_3\rangle^{s}}{i\xi_3\langle \xi\rangle^s} \widehat{v}(\xi_1) \widehat{v}(\xi_2) \widehat{v}(\xi_3) + \frac{\langle \xi\rangle^{2s}}{-i\xi} \widehat{v}(\xi) \widehat{v}(\xi) \widehat{v}(-\xi)\\
	&\hspace{-105pt}+ \sum_{\xi_3\neq 0} \frac{\langle\xi\rangle^s \langle \xi_3\rangle^{2s}\langle \xi\rangle^s}{i\xi_3\langle \xi\rangle^s} \widehat{v}(-\xi_3) \widehat{v}(\xi) \widehat{v}(\xi_3)+ \sum_{\xi_3\neq 0} \frac{\langle \xi\rangle^s \langle\xi_3\rangle^{2s}}{i\xi_3\langle \xi\rangle^s} \widehat{v}(\xi) \widehat{v}(-\xi_3) \widehat{v}(\xi_3)\\
	&\hspace{-150pt}= \sum_{\tiny \begin{array}{c}(\xi_1+\xi_2)(\xi_2+\xi_3)(\xi_3+\xi_1)\neq 0\\ \xi_1+\xi_2+\xi_3=\xi, \quad\xi_j \neq 0 \end{array} } \frac{\langle \xi_1\rangle^s \langle\xi_2\rangle^s \langle \xi_3\rangle^{s}}{i\xi_3\langle \xi\rangle^s} \widehat{v}(\xi_1) \widehat{v}(\xi_2) \widehat{v}(\xi_3) - \frac{\langle \xi\rangle^{2s}}{i\xi} |\widehat{v}|^2(\xi) \widehat{v}(\xi).
\end{align*}

The first term on the right side of above is the \emph{non-resonant} term, denoted $\mathcal{NR}(v,v,v)(\xi)$.  The second one is \emph{resonant} and is denoted $\mathcal{R}(v,v,v)(\xi)$.  Then we can rewrite \eqref{zperiodic} as 
$$
(\p_t + \p_{xxx}) z = -2 \mathcal{F}_{\xi}^{-1} \left[\mathcal{NR}(v,v,v)\right] + 2\mathcal{F}_{\xi}^{-1} \left[\mathcal{R}(v,v,v)\right].
$$

When $s$ is near $1/2$, the resonant term does not gain any derivatives since the oscillatory gain $(\xi_1 + \xi_2)(\xi_2+\xi_3)(\xi_3+\xi_1)$ is not available.  This is why the gain of derivatives vanish for $u-e^{-t\p_x^3}u_0$ in \cite{niko}.  To overcome this problem, we want to efficiently \emph{filter out} the roughest term.   To this end, we find an explicit solution for the IVP
$$
\left| \begin{array}{l}(\p_t+\p_{xxx})v_* = -2 \sum_{\xi \neq 0} \frac{\langle \xi \rangle^{2s}}{i\xi} |\widehat{v_*}(\xi)|^2 \widehat{v_*} (\xi) e^{i\xi x} \\
	v_*(0) = f \in L^2(\mathbf{T}). \end{array} \right.
$$

The solution to the IVP can be written explicitly as
\begin{equation}\label{defr}
R [f] (t,x) := \sum_{\xi \neq 0} \widehat{f}(\xi) e^{2i\frac{\langle \xi \rangle^{2s}}{\xi} |\widehat{f}(\xi)|^2 t} e^{i(\xi x + \xi^3 t)}.
\end{equation}

$R^*[u_0]$ in the statement of Theorem~\ref{thm2} corresponds to $\langle \nabla \rangle^s R[\langle \nabla \rangle^{-s} u_0]$.  For the method of such constructions, refer to \cite[Exercise 4.21]{tao2}.  Clearly, $R[f]$ is a unitary (non-linear) operator on $H^s$ for any $s\in \br$ and maps $H^s_x \to C^{0}_t H^s_x$.  

We perform another change of variable~$z = R[f] + y$ to filter out the roughest resonant term.  Then $y$ satisfies
\begin{align}\label{eqyper} 
(\p_t + \p_{xxx}) y &= -2\mathcal{F}_{\xi}^{-1} [\mathcal{NR}(R[f]+h+y,R[f]+h+y,R[f]+h+y)]\\
 &+2 \sum_{\xi\neq 0} \frac{\langle \xi \rangle^{2s}}{i\xi} B(\widehat{R[f]}(\xi), \widehat{h}(\xi), \widehat{y}(\xi)) e^{i\xi x}\notag 
 \end{align}
with the initial condition~$y(0) = -T(f,f)$, where $B(\al,\be,\ga):= |\al+\be+\ga|^2(\be+\ga) + \al|\be+\ga|^2+ \al^2 (\overline{\be+\ga}) + |\al|^2 (\be+\ga)$ for $\al,\be,\ga \in \mathbf{C}$.  The precise form of the polynomial~$B$ is not important. The main idea is that the Fourier coefficient~$B(\widehat{R[f]}(\xi), \widehat{h}(\xi), \widehat{y}(\xi))$ does not contain $|\wh{R[f]}(\xi)|^2 \wh{R[f]}(\xi)$. This is important heuristically since $R[f]$ is the least smooth term among the three.  Since at least one smooth term is present at any cubic resonant term, the cubic resonant term now should be as smooth as $y$ and $h$.

Now consider the non-resonant component in \eqref{eqyper}.  If $h$ and $y$ are smoother than $R[f]$, then the most problematic term should be $\mathcal{NR}(R[f],R[f],R[f])$.  Indeed, this is the only non-linearity in \eqref{eqyper} which obstructs the full derivative gain.  It can be shown that $\mathcal{NR}(R[f],R[f],R[f])$ only gains $1-s$~derivatives.  Filtering out this term by means of normal form will be enough to improve this to a full derivative gain.  We want to solve
\[
(\p_t + \p_{xxx}) y_* = 2\cf^{-1}_{\xi}[\mathcal{NR}(R[f],R[f],R[f])].
\]
We define the trilinear pseudo-differential operator~$J$ by $J(u,v,w)(x) := $
\[
-\f{2}{3}\sum_{\tiny \begin{array}{c} (\xi_1 + \xi_2)(\xi_2+ \xi_3)(\xi_3+\xi_1)\neq 0\\ \xi_1 + \xi_2 + \xi_3 = \xi, \qquad \xi_j \neq 0\end{array}} \f{\lan{\xi_1}^s \lan{\xi_2}^s \lan{\xi_3}^s}{i\xi_3 \lan{\xi}^s (\xi_1 + \xi_2)(\xi_2+ \xi_3)(\xi_3+\xi_1)} \wh{u}(\xi_1) \wh{v}(\xi_2) \wh{w}(\xi_3) e^{i\xi x}.
\]
Note that $J(\cdot,\cdot,\cdot)$ is symmetric in the first two variables.  Also, considering the algebraic identity 
\begin{equation}\label{eq:cubic}
(\tau_1 + \tau_2 +\tau_3) - (\xi_1+\xi_2+\xi_3)^3 = \sum_{j=1}^3 (\tau_j - \xi_j^3) -3(\xi_1 + \xi_2)(\xi_2+ \xi_3)(\xi_3+\xi_1),
\end{equation}
we can see that the Airy operator acts on $J$ as follows.
\begin{align*}
(\p_t + \p_{x}^3)J(u,v,w) =\, &J((\p_t + \p_x^3)u,v,w) + J(u,(\p_t + \p_x^3)v,w)+J(u,v,(\p_t + \p_x^3)w)\\
	 &-2 \cf^{-1}_{\xi}[\mathcal{NR}(u,v,w)].
\end{align*}

If we let $k := J(R[f], R[f], R[f])$, then by above
\begin{align*}
(\p_t + \p_{x}^3)k &=   2\cf^{-1}_{\xi}[\mathcal{NR}(u,v,w)]+ 2J(-2\cf^{-1}_{\xi}[\mathcal{R}(R[f],R[f],R[f])],R[f],R[f])\\
  &\quad + J(R[f],R[f], -2\cf^{-1}_{\xi}[\mathcal{R}(R[f],R[f],R[f])]).	
\end{align*}
with initial data $k(0) = J(f,f,f)$.  We write the Fourier coefficients of the quintilinear terms above.  For $\xi\neq 0$,
$\cf\left[J(\cf^{-1}_{\xi}[\mathcal{R}(R[f],R[f],R[f])],R[f],R[f])\right] (\xi) = $
\[
\sum_{\tiny \begin{array}{c} (\xi_1 + \xi_2)(\xi_2+ \xi_3)(\xi_3+\xi_1)\neq 0\\ \xi_1 + \xi_2 + \xi_3 = \xi, \qquad \xi_j \neq 0\end{array}} \f{\lan{\xi_1}^{3s} \lan{\xi_2}^s \lan{\xi_3}^s |\wh{f}(\xi_1)|^2}{-\xi_1 \xi_3  \lan{\xi}^s (\xi_1 + \xi_2)(\xi_2+ \xi_3)(\xi_3+\xi_1)} \wh{R[f]}(\xi_1) \wh{R[f]}(\xi_2) \wh{R[f]}(\xi_3).
\]
Similarly, $\cf\left[J(\cf^{-1}_{\xi}[\mathcal{R}(R[f],R[f],R[f])],R[f],R[f])\right](\xi)=$
\[
\sum_{\tiny \begin{array}{c} (\xi_1 + \xi_2)(\xi_2+ \xi_3)(\xi_3+\xi_1)\neq 0\\ \xi_1 + \xi_2 + \xi_3 = \xi, \qquad \xi_j \neq 0\end{array}} \f{\lan{\xi_1}^{s} \lan{\xi_2}^s \lan{\xi_3}^{3s} |\wh{f}(\xi_3)|^2}{-\xi_3^2  \lan{\xi}^s (\xi_1 + \xi_2)(\xi_2+ \xi_3)(\xi_3+\xi_1)} \wh{R[f]}(\xi_1) \wh{R[f]}(\xi_2) \wh{R[f]}(\xi_3).
\]
Note that above expressions are not genuinely quintilinear since there are only three different space-time Fourier modes present.  In this sense, estimates for these terms are essentially trilinear.

Writing $y = k + w$, we obtain our final equation
\begin{equation}\label{eqwper}
\left| \begin{array}{l} w_t + w_{xxx} = \mathcal{NR} + \mathcal{R} + \mathcal{Q}\\
w(0) = -T(f,f) - J(f,f,f)\end{array}\right.
\end{equation} 
where $\mathcal{NR}$, $\mathcal{R}$ and $\mathcal{Q}$ represent respectively remaining trilinear non-resonant terms, trilinear resonant terms and quintilinear terms as categorized below.

The non-resonant terms in $\mathcal{NR}$ are of the following type:
\begin{align}
&\mathcal{NR}(w, R[f]+h+k+w, R[f]+h+k+w); \label{eq:nr1}\\
&\mathcal{NR}(h+k, R[f]+h+k+w,R[f]+h+k+w);\label{eq:nr2}\\
&\mathcal{NR}(R[f]+h+k+w, R[f]+h+k+w, w);\label{eq:nr3}\\
&\mathcal{NR}(R[f]+h+k+w,R[f]+h+k+w, h+k).\label{eq:nr4}
\end{align}

The resonant term~$\mathcal{R}$ is same as the last term in \eqref{eqyper}, when $y$ is replaced with $k+w$.

The quintilinear terms in $\mathcal{Q}$ are of the following type:
\begin{align}
&J(\cf^{-1}_{\xi}[\mathcal{R}(R[f],R[f],R[f])],R[f],R[f]);\label{eq:q1}\\
&J(R[f],R[f], \cf^{-1}_{\xi}[\mathcal{R}(R[f],R[f],R[f])]).\label{eq:q2}
\end{align}

In the following section, we will examine in which functional spaces $R[f]$, $h$ and $k$ so that we can state the necessary estimates for $\mathcal{NR}$, $\mathcal{R}$ and $\mathcal{Q}$.

\subsection{Properties and regularity of $R[f]$, $h$ and $k$}\label{sec:properties}

We begin by examining an appropriate function space for $R[f]$.  We have already noted that $R[f]\in C^0_t L^2_x$.  The next lemma shows that $R[f] \in X_T^{0,b}$ for any $b\geq 0$, thus very close to the free solution~$e^{-t\p_x^3}f$.

\begin{lemma}\label{rffree}
Given $f\in L^2$, $s\leq \f{1}{2}$, $b\geq 0$ and $\eta \in \cs_t(\mathbf{R})$, we have
$$
\|\eta R[f]\|_{X^{0,b}} \lesssim \|\eta\|_{H^b} \max\left(\n{f}{L^2}, \n{f}{L^2}^{2b+1}\right)
$$
\end{lemma}

\begin{proof}
From \eqref{defr}, we have
$$
\wt{\eta \cdot R[f]}(\tau,\xi)= \wh{f}(\xi) \wh{\eta}\left(\tau - 2\f{\lan{\xi}^{2s}}{\xi} |\wh{f}(\xi)|^2 - \xi^3\right).
$$

Let $a_{\xi}:= 2\f{\lan{\xi}^{2s}}{\xi}|\wh{f}(\xi)|^2$,  then
\begin{align*}
\|\eta R[f]\|_{X^{0,b}} &= \n{\lan{\tau-\xi^3}^b \wh{f}(\xi) \wh{\eta}(\tau - a_{\xi} - \xi^3)}{L^2_{\tau}l^2_{\xi}}\\
	&\lesssim \n{\lan{\tau-a_{\xi} -\xi^3}^b \lan{a_{\xi}}^b \wh{f}(\xi) \wh{\eta}(\tau - a_{\xi} - \xi^3)}{l^2_{\xi} L^2_{\tau}}\\
	&\lesssim \n{\lan{\tau}^b \wh{\eta}(\tau)}{L^2_{\tau}} \n{\lan{a_{\xi}}^b \wh{f}(\xi)}{l^2_{\xi}}.
\end{align*}
Noting that $\sup_{\xi} |a_{\xi}| \lesssim \|f\|_{L^2}$, we have the desired estimate.
\end{proof}
Next, we establish the Lipschitz continuity of the map $R[f]$ on $L^2(\mathbf{T})$.

\begin{lemma}\label{lipschitz}
Let $R$ be defined as in \eqref{defr} with $s<1/2$ and $\gamma\in \mathbf{R}$.  Then for any $f, g\in L^2(\mathbf{T})$ with $f-g \in H^{\gamma}$,
$$
\n{R[f] - R[g]}{C^0_t H^{\gamma}_x([0,T]\times \mathbf{T})} \leq C_{N,T} \|f-g\|_{H^{\gamma}(\mathbf{T})}
$$
where $\|f\|_{L^2}+\|g\|_{L^2} < N$.
\end{lemma}

\begin{proof}
First we write $\widehat{f}(\xi) = |\widehat{f}(\xi)| e^{i \alpha_{\xi}}$ and $\widehat{g}(\xi) = |\widehat{g}(\xi)| e^{i \beta_{\xi}}$.  Denote $\theta_{\xi} :=\alpha_{\xi} - \beta_{\xi} $  Then, the Law of cosines, triangle and H\"older's inequality gives
\begin{align*}
\n{R[f] - R[g]}{C^0_t H^{\gamma}_x([0,T]\times \mathbf{T})} &= \sup_{t\in [0,T]} \n{\langle \xi\rangle^{\gamma} (|\widehat{f}| e^{2it\frac{\langle \xi\rangle^{2s}}{\xi}(|\widehat{f}|^2 - |\widehat{g}|^2 )+ i\theta_{\xi} } -  |\widehat{g}| )}{l^2_{\xi} (\mathbf{Z}\setminus \{0\})}\\
	&\hspace{-120pt}= \sup_{t\in [0,T]} \n{\langle \xi\rangle^{\gamma} \left( |\widehat{f}|^2 + |\widehat{g}|^2 - 2 |\widehat{f}| |\widehat{g}| \cos ( 2t\frac{\langle \xi\rangle^{2s}}{\xi}(|\widehat{f}|^2 - |\widehat{g}|^2 ) + \theta_{\xi} ) \right)^{\frac{1}{2}}}{l^2_{\xi} (\mathbf{Z}\setminus \{0\})}\\
	&\hspace{-120pt}\lesssim \|\langle \xi\rangle^{\gamma} (|\widehat{f}|-|\widehat{g}|)\|_{l^2_{\xi}} + 2 \sup_{t\in [0,T]} \n{\langle \xi\rangle^{2\gamma}  |\widehat{f}| |\widehat{g}|( 1 - \cos ( 2t\frac{\langle \xi\rangle^{2s}}{\xi}(|\widehat{f}|^2 - |\widehat{g}|^2 )+ \theta_{\xi} )}{l^1_{\xi} (\mathbf{Z}\setminus \{0\})}^{\frac{1}{2}}\\
	&\hspace{-120pt}\lesssim \|f-g\|_{H^{\gamma}} + 4\sup_{t\in [0,T]}  \n{\langle \xi\rangle^{2\gamma}  |\widehat{f}| |\widehat{g}| \sin^2 \left(t\frac{\langle \xi\rangle^{2s}}{\xi}(|\widehat{f}|^2 - |\widehat{g}|^2 )+ \theta_{\xi} \right)}{l^{1}_{\xi} (\mathbf{Z}\setminus \{0\})}^{\frac{1}{2}}.
\end{align*}
Using $\sin^2 (A+B) \lesssim A^2 + \sin^2 B$ and the assumption $s<1/2$, we need to estimate
\begin{align}
& \n{\langle \xi\rangle^{2\gamma}  |\widehat{f}| |\widehat{g}| (|\widehat{f}|^2 - |\widehat{g}|^2 )^2}{l^{1}_{\xi} (\mathbf{Z}\setminus \{0\})}^{\frac{1}{2}},\label{lip1}\\
&\n{\langle \xi\rangle^{2\gamma}  |\widehat{f}| |\widehat{g}| \sin^2 \theta_{\xi}}{l^{1}_{\xi} (\mathbf{Z}\setminus \{0\})}^{\frac{1}{2}}.\label{lip2}
\end{align}
The bound for \eqref{lip1} is straight-forward.  By H\"older's and triangle inequalty,
$$
\eqref{lip1} \lesssim \|f\|_{L^2}^{\frac{1}{2}} \|g\|_{L^2}^{\frac{1}{2}} \n{\langle \xi \rangle^{\gamma} (\widehat{f} - \widehat{g}) (|\widehat{f}| + |\widehat{g}|)}{l^{\infty}_{\xi}} \lesssim \|f\|_{L^2}^{\frac{3}{2}} \|g\|_{L^2}^{\frac{3}{2}} \|f-g\|_{H^{\gamma}}.
$$

For \eqref{lip2}, we apply the Law of sines.  Without loss of generality, we can assume $\theta_{\xi} \in (0,\pi)$.  Noting that the triangle with side-lengths equal to $|\widehat{f}|, |\widehat{g}|, |\widehat{f} - \widehat{g}|$ has the angle~$\theta_{\xi}$ which is opposite to the side with length $|\widehat{f}-\widehat{g}|$, we can deduce that $|\widehat{f}|\sin \theta_{\xi}\leq |\widehat{f} - \widehat{g}|$ and likewise for $|\widehat{g}|$.  Thus,
$$
\eqref{lip2} \leq \n{\langle \xi\rangle^{2\gamma} |\widehat{f} - \widehat{g}|^2 }{l^1_{\xi}}^{\frac{1}{2}} \sim \n{\langle \xi \rangle^{\gamma} |\widehat{f} - \widehat{g}|}{l^2_{\xi}} \sim \|f-g\|_{H^{\gamma}}.
$$
\end{proof}

We take a detour to make a remark about the uniform discontinuity of $R[\cdot]$ when $s>1/2$ (equivalently, uniform discontinuity of $R^*[\cdot]$ in $H^{-s}$).  The following remark serves as a heuristic evidence that $R^*[f]$ dominates the evolution of the periodic KdV below $H^{-1/2}$.

\begin{remark}\label{rem:nonuniform}
For any $\ve>0$, $R[\cdot]: L^2 \to L^{\infty}_t([0,\ve]; L^2_x)$ is not uniformly continuous on a bounded set of $L^2(\bt)$.
\end{remark}

\begin{proof}
Given any $0<\de, \ve \ll 1$, we show that there exists $f,g \in L^2$ such that $\|f\|_{L^2}+\|g\|_{L^2}\leq 2$ and $\|f-g\|_{L^2} \leq \de$, but $\sup_{t\in [0,\ve]} \|R[f](t)-R[g](t)\|_{L^2_x} \geq 1$.

For some $\xi\in \mathbf{Z}\setminus\{0\}$ to be determined, define $f(x) := e^{ix\xi}$ and $g(x) := (1-\de) e^{ix\xi}$.  Then applying the law of cosines,
\begin{align*}
\n{R[f](t) - R[g](t)}{L^2}^2 &= \left| e^{2i\f{\lan{\xi}^{2s}}{\xi} t + \xi^3 t} - (1-\de) e^{2i\f{\lan{\xi}^{2s}}{\xi} |1-\de|^2 t + \xi^3 t} \right|^2\\
	&= 1 + (1-\de)^2  - 2(1-\de) \cos\left(2\f{\lan{\xi}^{2s}}{\xi}(1-(1-\de)^2)t\right) 
\\
	&= \de^2 + 2(1-\de)\left[ 1- \cos \left(2\f{\lan{\xi}^{2s}}{\xi}(2\de -\de^2)t\right) \right]\\
		&= \de^2 + 4(1-\de)\sin^2\left(\f{\lan{\xi}^{2s}}{\xi}\de(2-\de)t\right) \geq 2\sin^2\left(\f{\lan{\xi}^{2s}}{\xi}\de(2-\de)t\right).
\end{align*}

If $s>1/2$, choose $\xi\in \mathbf{Z}$ so that $\f{\lan{\xi}^{2s}}{\xi}\de(2-\de)t = \arcsin (1/\sqrt{2})$ for some $t\in [0,\ve]$.
\end{proof}

Next, we consider mapping properties of the bilinear operator~$T$, which will give us the appropriate functional space for $h$ as well as the necessary regularity for the initial data $w(0)$.

The following lemma implies that $T(f,f) \in H^1(\mathbf{T})$ and also that $h\in C^0_t([0,T]; H^1_x) \subset X^{1,0}_T$ for any $T>0$.

\begin{lemma}\label{tpmap1}
$T: L^2(\mathbf{T}) \times L^2(\mathbf{T}) \to H^{1}(\mathbf{T})$ is a bounded bilinear operator.
\end{lemma}

\begin{proof}
Let $u,v\in C^{\infty}(\mathbf{T})$.  Then
\begin{equation}\label{tpeq}
\| T(u,v)\|_{H^{1}} \sim \n{\sum_{\xi_1 (\xi-\xi_1)\neq 0} \frac{\langle\xi_1\rangle^s \langle\xi-\xi_1\rangle^s \langle\xi\rangle^{1-s} }{\xi_1 (\xi-\xi_1)} \widehat{u}(\xi_1) \widehat{v}(\xi-\xi_1)}{l^2_{\xi}(\mathbf{Z}\setminus \{0\}) }.
\end{equation}

By symmetry, we can assume $|\xi_1| \geq |\xi-\xi_1|$.  Then by H\"older and Sobolev embedding,
\begin{align*}
\eqref{tpeq} &\lesssim  M_{\ve} \n{\sum_{\xi_1 \neq \xi} |\widehat{u}|(\xi_1) \frac{|\widehat{v}|(\xi-\xi_1)}{|\xi-\xi_1|^{\frac{1}{2}+\varepsilon}}}{l^2_{\xi}(\mathbf{Z}) }\\
	 &\lesssim M_{\ve} \n{\mathcal{F}^{-1}[ |\widehat{u}|] |\p_x |^{-\frac{1}{2}-\varepsilon} \mathcal{F}^{-1}[|\widehat{v}|]}{L^2_{x}(\mathbf{T})} \lesssim_{\varepsilon} M_{\ve} \| u\|_{L^2(\mathbf{T})} \| v \|_{L^2_{x}(\mathbf{T})}
\end{align*}
where we take $\ve>0$ small and
$$
M_{\ve} := \sup_{\xi \xi_1 (\xi-\xi_1)\neq 0}\frac{\langle\xi_1\rangle^s \langle\xi-\xi_1\rangle^s \langle\xi\rangle^{1-s} }{|\xi_1| |\xi-\xi_1|^{\frac{1}{2}-\varepsilon}}.
$$

It is easy to see that $M$ is a bounded quantity if $s<1/2$, thus the claim follows.
\end{proof}

The following lemma gives an estimate on the term~$h$.

\begin{lemma}\label{le:tspace}
Let $u,v \in X^{0,\f{1}{2}}\cap \wt{l^2_{\xi} L^1_{\tau}}$.  Then
\[
\|T(u,v)\|_{X^{0, \f{1}{2}}} \lesssim \|u\|_{X^{0,\f{1}{2}}\cap \wt{l^2_{\xi} L^1_{\tau}}} \|v\|_{X^{0,\f{1}{2}}\cap \wt{l^2_{\xi} L^1_{\tau}}}.
\] 
\end{lemma}
\textbf{Remark:} By interpolating this estimate with Lemma~\ref{tpmap1}, we also obtain for $b\in [0,1/2]$
\begin{equation}\label{eq:tspace}
\|T(v,v)\|_{X^{1-2b, b}_1} \lesssim  \|v\|^2_{X^{0,\f{1}{2}}\cap \wt{l^2_{\xi} L^1_{\tau}}}.
\end{equation}
\begin{proof}
Recall that $\n{T(u,v)}{X^{0, \f{1}{2}}} =$
\begin{equation}\label{eq:tuv}
\n{\lan{\tau-\xi^3}^{\f{1}{2}} \int_{\mathbf{R}}\sum_{\xi_1(\xi-\xi_1) \neq 0} \frac{\langle\xi_1\rangle^s \langle\xi -\xi_1\rangle^s}{\xi_1 (\xi - \xi_1)\lan{\xi}^{s}} \wt{u}(\tau_1, \xi_1) \wt{v} (\tau - \tau_1,\xi-\xi_1) \, d\tau_1}{L^2_{\tau} l^2_{\xi}}
\end{equation}
Since $[\tau_1 - \xi_1^3] + [(\tau-\tau_1) - (\xi-\xi_1)^3] = [\tau - \xi^3] + 3\xi \xi_1 (\xi-\xi_1)$, we write
\begin{equation}\label{eq:tauxi}
\lan{\tau-\xi^3}^{\f{1}{2}} \lesssim \lan{\tau_1 - \xi_1^3}^{\f{1}{2}} + \lan{(\tau-\tau_1) -(\xi-\xi_1)^3}^{\f{1}{2}} + \lan{\xi}^{\f{1}{2}} \lan{\xi_1}^{\f{1}{2}} \lan{\xi-\xi_1}^{\f{1}{2}}.
\end{equation}
Since \eqref{eq:tuv} is symmetric in $\xi_1, \xi-\xi_1$ variable, we ignore the middle term on RHS of \eqref{eq:tauxi}.

To estimate the component~$I_1$ of \eqref{eq:tuv} containing $\lan{\tau_1 -\xi_1^3}^b$, note 
\[
\sup_{\xi \xi_1 (\xi-\xi_1)\neq 0} \f{\lan{\xi_1}^s \lan{\xi-\xi_1}^s }{|\xi_1| |\xi-\xi_1|^{s}\lan{\xi}^{s}} <\infty
\]
 
Thus we apply Plancherel, H\"older's and Sobolev embedding to obtain
\begin{align*}
I_1 &\lesssim  \n{\int_{\mathbf{R}} \sum_{\xi-\xi_1\neq 0} \lan{\tau_1 -\xi_1^3}^{\f{1}{2}} |\wt{u}|(\tau_1, \xi_1) \f{1}{|\xi-\xi_1|^{1-s}}|\wt{v}|(\tau-\tau_1, \xi-\xi_1) \, d\tau_1}{L^2_{\tau} l^2_{\xi}}\\
	&\lesssim \n{\lan{\tau -\xi^3}^{\f{1}{2}} |\wt{u}|}{L^2_{\tau}l^2_{\xi}} \n{\f{1}{|\nabla|^{1-s}} \cf^{-1}_{\tau,\xi} [|\wt{v}|]}{L^{\infty}_{t} L^{\infty}_{x}}\\
	&\lesssim_s \n{u}{X^{0,\f{1}{2}}} \n{\cf^{-1}_{\tau,\xi}[|\wt{v}|]}{L^{\infty}_t L^2_x} \lesssim \n{u}{X^{0,\f{1}{2}}} \n{\wt{v}}{l^2_{\xi} L^1_{\tau}}.
\end{align*}

The other component~$I_2$ containing the last term of \eqref{eq:tuv} is more direct.
\begin{align*}
I_2 &\lesssim \n{\int_{\mathbf{R}}\sum_{\xi_1(\xi-\xi_1) \neq 0} \frac{\langle\xi_1\rangle^{s+\f{1}{2}} \langle\xi -\xi_1\rangle^{s+\f{1}{2}}\lan{\xi}^{\f{1}{2}-s}}{\xi_1 (\xi - \xi_1)} |\wt{u}|(\tau_1, \xi_1) |\wt{v}| (\tau - \tau_1,\xi-\xi_1) \, d\tau_1}{L^2_{\tau} l^2_{\xi}}\\
&\lesssim \sup_{\xi_1(\xi-\xi_1)\neq 0} \frac{\langle\xi_1\rangle^{s+\f{1}{2}} \langle\xi -\xi_1\rangle^{s+\f{1}{2}}\lan{\xi}^{\f{1}{2}-s}}{|\xi_1| |\xi - \xi_1|} \n{\int_{\mathbf{R}}\sum_{\xi_1(\xi-\xi_1) \neq 0}  |\wt{u}|(\tau_1, \xi_1) |\wt{v}| (\tau - \tau_1,\xi-\xi_1) \, d\tau_1}{L^2_{\tau} l^2_{\xi}}\\
	&\lesssim \| \cf^{-1}_{\tau,\xi}[|\wt{u}|] \|_{L^4_{t,x}} \|\cf^{-1}_{\tau,\xi}[|\wt{u}|]\|_{L^4_{t,x}} \lesssim \|u\|_{X^{0,\f{1}{3}}} \|v\|_{X^{0,\f{1}{3}}}
\end{align*}
where the last inequality follows from \eqref{xsb3}.  Thus we have the claim.
\end{proof}

Finally, we consider mapping properties of the trilinear operator~$J$ and derive an appropriate function space for $k$.  The following lemma implies that $J(f,f,f)\in H^1(\bt)$ and also that $k \in C^0_t([0,T];H^1_x)\subset X^{1,0}_T$.

\begin{lemma}\label{le:jmap}
$J: L^2(\bt) \times L^2(\bt)\times L^2(\bt) \to H^1(\bt)$ is a bounded trilinear operator.
\end{lemma}

\begin{proof}
Let $u,v,w\in C^{\infty}(\bt)$.  Then $\|J(u,v,w)\|_{H^1}\sim $
\[
\n{\sum_{\tiny \begin{array}{c} (\xi_1 + \xi_2)(\xi_2+ \xi_3)(\xi_3+\xi_1)\neq 0\\ \xi_1 + \xi_2 + \xi_3 = \xi, \qquad \xi_j \neq 0\end{array}} \f{\lan{\xi}^{1-s} \lan{\xi_1}^s \lan{\xi_2}^s \lan{\xi_3}^s}{i\xi_3 (\xi_1 + \xi_2)(\xi_2+ \xi_3)(\xi_3+\xi_1)} \wh{u}(\xi_1) \wh{v}(\xi_2) \wh{w}(\xi_3)}{l^2_{\xi}}.
\]
Note that the above expression is symmetric in $\xi_1$, $\xi_2$ and also that the estimate will be easier when $|\xi_3| \sim \max(|\xi|, |\xi_1|, |\xi_2|, |\xi_3|)$.  Thus, we can assume without loss of generality that $|\xi|\sim |\xi_1| \gtrsim \max(|\xi_3|, |\xi_2|)$.  Then, assuming that the summation is restricted to the set $\{(\xi_1, \xi_2,\xi_3)\in \mathbf{Z}^3: (\xi_1 + \xi_2)(\xi_2+ \xi_3)(\xi_3+\xi_1)\neq 0,\, \xi_1 + \xi_2 + \xi_3 = \xi,\, \xi_1\xi_2\xi_3 \neq 0\}$
\[
\|J(u,v,w)\|_{H^1} \lesssim \n{\sum_{|\xi|\sim |\xi_1| \gtrsim \max(|\xi_3|, |\xi_2|)} \f{\lan{\xi}^{1} \lan{\xi_2}^s |\wh{u}|(\xi_1) |\wh{v}|(\xi_2) |\wh{w}|(\xi_3)}{\lan{\xi_3}^{1-s} |\xi_1 + \xi_2||\xi_2+ \xi_3||\xi_3+\xi_1|} }{l^2_{\xi}}.
\]
We split the sum into two pieces.

\begin{itemize}
\item[Case 1.] $\xi_1 \sim \xi_2 \sim \xi_3 \sim \xi$.  Note that in this case 
\begin{equation}\label{eq:xigain}
|\xi_1 + \xi_2||\xi_2+ \xi_3||\xi_3+\xi_1| \gtrsim |\xi| \min(|\xi_1 + \xi_2|, |\xi_2+ \xi_3|, |\xi_3+ \xi_1|)^2.
\end{equation}
Splitting the sum again into three pieces, we apply Young's and H\"older's inequalities to obtain
\begin{align*}
&\n{\sum_{|\xi_1 + \xi_2| \leq |\xi_2+\xi_3|, |\xi_3+\xi_1|} \lan{\xi_1 + \xi_2}^{-2}|\wh{u}|(\xi_1) |\wh{v}|(\xi_2) |\wh{w}|(\xi_3) }{l^2_{\xi}}\\
 &\hspace{30pt}\lesssim \n{ \sum_{\xi_3\in \mathbf{Z}} |\wh{w}|(\xi_3) \lan{\xi-\xi_3}^{-2} \left[ |\wh{u}|*_{\xi_1} |\wh{v}|\right](\xi-\xi_3)}{l^2_{\xi}}\\
&\hspace{30pt}\lesssim \n{ |\wh{w}|}{l^2_{\xi}} \n{ \lan{\xi}^{-2} \left[ |\wh{u}|*_{\xi_1} |\wh{v}|\right]}{l^1_{\xi}} \lesssim \n{\lan{\xi}^{-2}}{l^1_{\xi}} \n{\wh{w}}{l^2_{\xi}}  \n{\wh{u}}{l^2_{\xi}} \n{\wh{v}}{l^{2}_{\xi}}.
\end{align*}

\item[Case 2.]  If $\xi\sim \xi_1 \gg \xi_3$, then
\begin{equation}\label{eq:xisquaregain}
|\xi_1 + \xi_2||\xi_2+ \xi_3||\xi_3+\xi_1| \gtrsim |\xi|^2 \min(|\xi_1 + \xi_2|,\, |\xi_2+ \xi_3|).
\end{equation}
By splitting the sum, and we can assume $|\xi_1 + \xi_2|\geq |\xi_2+ \xi_3|$.  Apply H\"older's and Young's inequalities to obtain
\begin{align*}
&\n{\lan{\xi}^{-1+s}\sum_{\xi_1\in \mathbf{Z}}|\wh{u}|(\xi_1)\lan{\xi-\xi_1}^{-1} \left[ |\wh{v}|*_{\xi_2} |\wh{w}|\right](\xi-\xi_1)\ }{l^2_{\xi}}\\
&\hspace{30pt}\lesssim \n{\xi^{-1+s}}{l^2_{\xi}} \n{ \wh{u}}{l^2_{\xi}} \n{\lan{\xi}^{-1} \left[ |\wh{v}|* |\wh{w}|\right]}{l^2_{\xi}}\\
&\hspace{30pt}\lesssim  \n{\xi^{-1+s}}{l^2_{\xi}} \n{\lan{\xi}^{-1}}{l^2_{\xi}} \n{ \wh{u}}{l^2_{\xi}}\n{  |\wh{v}|* |\wh{w}|}{l^{\infty}_{\xi}} \lesssim_{s} \n{ \wh{u}}{l^2_{\xi}}\n{ \wh{v}}{l^2_{\xi}}\n{ \wh{w}}{l^2_{\xi}}.
\end{align*}
\end{itemize}
This concludes the proof.
\end{proof}

The following lemma gives the appropriate space for $k$.

\begin{lemma}\label{le:kspace}
Let $u,v,w \in X^{0,1+2\de}$ for $0<10\de<1-2s$.  Then
\[
\|J(u,v,w)\|_{X^{0,\f{1+\de}{2}}} \lesssim_{\de} \|u\|_{X^{0,1+2\de}} \|v\|_{X^{0,1+2\de}}\|w\|_{X^{0,1+2\de}}.
\]
\end{lemma}
\textbf{Remark:} It is likely that the bounded property $J: \left(X^{0,\f{1}{2}}\cap \wt{l^2_{\xi} L^1_{\tau}}\right)^3 \to X^{0,\f{1}{2}}$ also holds.  But Lemma~\ref{le:kspace} will suffice for our purpose since the only argument of $J(\cdot, \cdot,\cdot)$ is $R[f]$, which is smooth in the modulation frequency as stated in Lemma~\ref{rffree}.  

In fact, the proof will imply that
\begin{equation}\label{eq:jspace}
\|J(R[f],R[f],R[f])\|_{X_T^{0, \f{1+\de}{2}}} \lesssim \|R[f]\|_{X_T^{0,1+2\de}}\|R[f]\|^2_{X_T^{0,\f{1}{2}+\de}}.
\end{equation}

Considering \eqref{xsb1} and Lemma~\ref{rffree}, we can interpolate above with Lemma~\ref{le:jmap} to obtain $k \in X_T^{1-\f{2b}{1+\de},\, b}$ for all $b\in [0,\f{1+\de}{2}]$ for a fixed $\de>0$ small.

\begin{proof}
Recall that $\|J(u,v,w)\|_{X^{0,\f{1+\de}{2}}} = $
\begin{equation*}
\n{\sum_{\tiny \begin{array}{c} (\xi_1 + \xi_2)(\xi_2+ \xi_3)(\xi_3+\xi_1)\neq 0\\ \xi_1 + \xi_2 + \xi_3 = \xi, \qquad \xi_j \neq 0\end{array}} \f{\lan{\tau-\xi^3}^{\f{1+\de}{2}} \lan{\xi_1}^s \lan{\xi_2}^s \lan{\xi_3}^s}{i\xi_3 \lan{\xi}^{s}(\xi_1 + \xi_2)(\xi_2+ \xi_3)(\xi_3+\xi_1)} \wh{u}(\xi_1) \wh{v}(\xi_2) \wh{w}(\xi_3)}{L^2_{\tau} l^2_{\xi}}.
\end{equation*}

For this trilinear estimate, we use the embedding \eqref{xsb2}.
First we localize each variable in terms of its dispersive frequencies, i.e. $\langle \tau_j- \xi_j^3\rangle \sim L_j$ for $j=1,2,3$ and $\langle \tau-\xi^3\rangle \sim L$, where $L, L_j\gtrsim 1$ are dyadic indices.  We only need to insure that the final estimate includes $L_{\max}^{-\varepsilon}$ for some $\varepsilon>0$ so that sum in these indices and also gain a small positive power of $T$ if necessary.
 
Consider the algebraic identity~\eqref{eq:cubic}.  Note that we must have either that $L \lesssim |\xi_1+ \xi_2||\xi_2+\xi_3||\xi_3+\xi_1|$ or $L \lesssim L_{j}$ for some $j=1,2,3$. 

First, we consider when $L \lesssim |\xi_1+ \xi_2||\xi_2+\xi_3||\xi_3+\xi_1|$. Apply Plancherel and H\"older followed by \eqref{xsb2} to obtain

\begin{align*}
\|J(u,v,w)\|_{X^{0,\f{1+\de}{2}}} &\lesssim M \| \widetilde{u_{-\delta}} * ( \widetilde{v_{-\delta}} * \widetilde{w_{-\delta}})\|_{L^2_{\tau}l^2_{\xi}} \sim M \|  u_{-\delta}v_{-\delta}w_{-\delta}\|_{L^2_{t,x}}\\
	 &\lesssim M \| u_{-\delta}\|_{L^6_{t,x}}  \| v_{-\delta}\|_{L^6_{t,x}}  \| w_{-\delta}\|_{L^6_{t,x}}  \lesssim_{\delta} M \|v\|_{X^{0,\frac{1}{2}+\de}}\|v\|_{X^{0,\frac{1}{2}+\de}}\|v\|_{X^{0,\frac{1}{2}+\de}}
\end{align*}
where $\widetilde{v_{-\delta}} (\tau,\xi) := \langle \xi \rangle^{-\delta} \lan{\tau-\xi^3}^{\de/2} |\widetilde{v}|(\tau,\xi)$ and 
\begin{equation*}
M := \sup_{\tiny \begin{array}{c}(\xi_1+\xi_2)(\xi_2+\xi_3)(\xi_3+\xi_1)\neq 0\\ \xi_1+\xi_2+\xi_3=\xi,\quad \xi_j \neq 0 \end{array}} \frac{\lan{\xi_1}^{s+\de} \lan{\xi_2}^{s+\de} }{L_{\max}^{\f{\de}{2}} \lan{\xi_3}^{1-s-\de}\lan{\xi}^{s} |(\xi_1 + \xi_2)(\xi_2+ \xi_3)(\xi_3+\xi_1)|^{\f{1-2\de}{2}}}.
\end{equation*}

Thus, it suffices to show that there exists an absolute upper bound $C$ for
\begin{equation}\label{eq:mbound}
\frac{\lan{\xi_1}^{s+\de} \lan{\xi_2}^{s+\de} }{\lan{\xi_3}^{1-s-\de}\lan{\xi}^{s} |(\xi_1 + \xi_2)(\xi_2+ \xi_3)(\xi_3+\xi_1)|^{\f{1-2\de}{2}}}
\end{equation}
for all $\xi\in \mathbf{Z}\setminus \{0\}$ and $\xi_j$ for $j=1,2,3$ satisfying the same restrictions as before.

As observed in \eqref{eq:xigain} and \eqref{eq:xisquaregain}, the weight $(\xi_1 + \xi_2)(\xi_2+ \xi_3)(\xi_3+\xi_1)$ gains $\xi_{\max}^2$ most of the time, unless $\xi_1 \sim \xi_2 \sim \xi_3$ in which case it still gains $\xi_{\max}$.

If $\xi_1 \sim \xi_2 \sim \xi_3$, then \eqref{eq:mbound} is bounded by $C\xi_{\max}^{3s+4\de -\f{3}{2}}$.  Thus letting $4\de < \f{3}{2} - 3s$ suffice to bound \eqref{eq:mbound}.  Otherwise, \eqref{eq:mbound} is bounded by $C\xi_{\max}^{2s+3\de - 1}$, which is bounded if $3\de \leq 1-2s$.

Now, it remains to estimate the case $L \gg |\xi_1+\xi_2| |\xi_2 + \xi_3| |\xi_3+ \xi_1|$.  As noted previously, this forces $L \lesssim L_j$ for some $j=1,2,3$ due to \eqref{eq:cubic}.  Without loss of generality, assume $L \lesssim L_1=L_{\max}$. 
\begin{align*}
\|J(u,v,w)\|_{X^{0,\f{1+\de}{2}}} &\lesssim M' \| \lan{\xi}^{-\de} L^{\f{1+\de}{2}}  \widetilde{u_{-\delta}} * ( \widetilde{v_{-\delta}} * \widetilde{w_{-\delta}})\|_{L^2_{\tau}l^2_{\xi}}\\
	&\hspace{-50pt}\sim M'  L_1^{1+\de} L^{-\f{1+\de}{2}} \sup_{\|z\|_{L^2_{\tau}l^2_{\xi}}=1} \left| \int_{\tau_1+\tau_2+\tau_3=\tau} \sum_{\xi_1+\xi_2+\xi_3=\xi} \widetilde{u_{-\de}}\, \widetilde{v_{-\de}}\, \widetilde{w_{-\de}}\, \lan{\xi}^{-\de} z(\tau,\xi) d\sigma\right| \\
	&\hspace{-50pt} \lesssim M' L_1^{1+\de}\|u\|_{L^2_t l^2_{\tau}} \sup_{\|z\|_{L^2_{\tau}l^2_{\xi}}=1} \n{ (L^{-\f{1+\de}{2}}\lan{\cdot}^{-\de}z) * ( \widetilde{v_{-\delta}}*\widetilde{w_{-\delta}})}{L^2_{\tau} l^2_{\xi}}\\
	&\hspace{-50pt}\lesssim M' L_{\max}^{-\de} \|u\|_{X^{0,1+2\delta}} \sup_{\|z\|_{L^2_{\tau}l^2_{\xi}}=1} \n{\mathcal{F}_{\tau,\xi}^{-1} \left[\lan{\tau-\xi^3}^{-\frac{1+\de}{2}}\langle \xi \rangle^{-\delta}z\right] v_{-\delta} w_{-\delta}}{L^2_{t,x}}\\
	&\hspace{-50pt}\lesssim_{\delta} M' \|u\|_{X^{0,1+\delta}}\|v\|_{X^{0,\frac{1}{2}+\delta}}\|w\|_{X^{0,\frac{1}{2}+\delta}}
\end{align*}
where we have used H\"older and \eqref{xsb2} for the last inequality, and
\[
M' := \sup_{\tiny \begin{array}{c}(\xi_1+\xi_2)(\xi_2+\xi_3)(\xi_3+\xi_1)\neq 0\\ \xi_1+\xi_2+\xi_3=\xi,\quad \xi_j \neq 0 \end{array}} \frac{\lan{\xi_1}^{s+\de} \lan{\xi_2}^{s+\de} }{ \lan{\xi_3}^{1-s-\de}\lan{\xi}^{s-\de} |(\xi_1 + \xi_2)(\xi_2+ \xi_3)(\xi_3+\xi_1)|}.
\]
It is easy to see that $M' \lesssim \eqref{eq:mbound}$, so we obtain the desired estimate.
\end{proof}

Finally, based on Lemmas~\ref{rffree}-\ref{le:kspace}, we assign  appropriate function spaces for $R[f]$, $h$, $k$ as follows:
\[
R[f] \in X^{0,1+2\de}_T; \qquad h \in C^0_t H^1_x \cap \left(\cap_{b\in [0,\f{1}{2}]} X_T^{1-2b, b}\right); \qquad k \in \cap_{b\in [0,\f{1+\de}{2}]} X_T^{1-\f{2b}{1+\de}, b}
\]
where $0<\de <(1-2s)/10$ is fixed.  Note that the function space for $k$ continuously embeds in the space for $h$, i.e. $k$ is more regular than $h$.  In our multilinear estimate, we will assume both $h,k \in X^{1, 0} \cap X^{0,\f{1}{2}}$ and $R[f]\in X^{0,\f{1}{2}}$ to reduce the number of necessary estimates.

\subsection{Estimates for the non-linearities.}\label{sec:nonlinearity}
Consider the IVP~\eqref{eqwper}.  We want to give a contraction argument for $w$ in $Y^{\ga, \f{1}{2}+\de}$.  First consider the nonlinearity~$\mathcal{NR}$.  The estimates \eqref{eq:wff}-\eqref{eq:ffh} will be useful in managing the terms \eqref{eq:nr1}-\eqref{eq:nr4} respectively.

\begin{lemma}\label{nonres}
Let $u,v,w$ be in appropriate function spaces.  Given any $0\leq s <1/2$ and $0<10\de < 1-2s$, we have for all $0<\ga <1-10\de$
\begin{align}
\n{\mathcal{F}_{\xi}^{-1}[\mathcal{NR}(u,v,w)] }{X_T^{\ga,-\frac{1}{2}+\delta}} &\lesssim_{\delta} T^{\de} \|u\|_{X^{\ga,\f{1}{2}+\de}} \|v\|_{X^{0,\frac{1}{2}}} \|w\|_{X^{0,\f{1}{2}}};\label{eq:wff}\\
\n{\mathcal{F}_{\xi}^{-1}[\mathcal{NR}(u,v,w)] }{X_T^{\ga,-\frac{1}{2}+\delta}} &\lesssim_{\delta} T^{\de} \|u\|_{X^{1, 0} \cap X^{0,\f{1}{2}}} \|v\|_{X^{0,\frac{1}{2}}} \|w\|_{X^{0,\f{1}{2}}}; \label{eq:hff}\\
\n{\mathcal{F}_{\xi}^{-1}[\mathcal{NR}(u,v,w)] }{X_T^{\ga,-\frac{1}{2}+\delta}} &\lesssim_{\delta} T^{\de} \|u\|_{X^{0,\f{1}{2}}} \|v\|_{X^{0,\frac{1}{2}}} \|w\|_{X^{\ga,\f{1}{2}+\de}};\label{eq:ffw}\\
\n{\mathcal{F}_{\xi}^{-1}[\mathcal{NR}(u,v,w)] }{X_T^{\ga,-\frac{1}{2}+\delta}} &\lesssim_{\delta} T^{\de} \|u\|_{X^{0,\f{1}{2}}} \|v\|_{X^{0,\frac{1}{2}}} \|w\|_{X^{1, 0} \cap X^{0,\f{1}{2}}}.\label{eq:ffh}
\end{align}
\end{lemma}
\textbf{Remark:} Note that one may choose any $\ga<1$ for the estimates above.  However, when one reduces the size of $\de$, one is penalized by shorter contraction time which in turn leads to a larger growth bound.  Also the implicit constant due to \eqref{xsb2} goes to infinity according to Bourgain, \cite{Bour}.  It appears that the growth bound explodes as $\ga \nearrow 1$.

\begin{proof}
Since this is a trilinear estimate, we will need to use the $L^6$~embedding~\eqref{xsb2} in order to maximize our gain.  However, this has the disadvantage when one wants to produce a large power of $T$.  One can either sacrifice the derivative gain to obtain a better polynomial-in-time growth bound or sacrifice the better growth bound to obtain the maximal derivative gain.  The goal of this paper better aligns with the latter path.

As in the proof of Lemma~\ref{le:kspace}, we localize each variable in terms of its modulation frequencies so that $\langle \tau_j- \xi_j^3\rangle \sim L_j$ for $j=1,2,3$ and $\langle \tau-\xi^3\rangle \sim L$, where $L, L_j\gtrsim 1$ are dyadic indices.  We aim to estimate each localized components by $C L_{\max}^{-\de}$ for a constant~$C=C(\de)$.  This will suffice in summing all the component as well as gaining $T^{\de}$ factor (one can accomplish both by picking $\de': 0<10\de <10\de' <1-2s$.  We remark that the constant $10$ in front of $\de$ is somewhat arbitrary and is subject to improvements.

Recall $L_{\max} \gtrsim |(\xi_1+ \xi_2)(\xi_2+\xi_3)(\xi_3+\xi_1)|$ due to the algebraic identity~\eqref{eq:cubic}.  We will use this fact throughout the proof.

\textbf{Proof of \eqref{eq:wff}:}  If $L\sim L_{\max}$, we have $\n{\mathcal{F}_{\xi}^{-1}[\mathcal{NR}(u,v,w)] }{X_T^{\ga,-\frac{1}{2}+\delta}}= $
\begin{align*}
& \n{ \sum_{\tiny \begin{array}{c}(\xi_1+\xi_2)(\xi_2+\xi_3)(\xi_3+\xi_1)\neq 0\\ \xi_1+\xi_2+\xi_3=\xi, \quad \xi_j \neq 0 \end{array} } \frac{\langle \xi_1\rangle^s \langle\xi_2\rangle^s \langle \xi_3\rangle^{s}\langle \xi\rangle^{\ga-s}}{i\xi_3\langle\tau-\xi^3 \rangle^{\frac{1}{2}-\delta}} [\widetilde{u}(\xi_1)*_{\tau} \widetilde{v}(\xi_2)*_{\tau} \widetilde{w}(\xi_3)](\tau)}{L^2_{\tau} l^2_{\xi}(\mathbf{Z}\setminus \{0\})}\\
	&\lesssim_{\delta} M_1 \n{ \left[\lan{\cdot}^{\ga-\de}|\wt{u}|\right] * ( \widetilde{v^{-\delta}} * \widetilde{w^{-\de}})}{L^2_{\tau}l^2_{\xi}} \sim M_1 \n{ \left[\lan{\na}^{\ga-\de} \cf^{-1}_{\tau,\xi}[|\wt{u}|]\right] \, v^{-\de} \,w^{-\de} }{L^2_{t,x}}\\
	&\lesssim M_1 \n{ \left[\lan{\na}^{\ga-\de} \cf^{-1}_{\tau,\xi}[|\wt{u}|]\right] }{L^6_{t,x}}  \| v^{-\de}\|_{L^6_{t,x}} \| w^{-\de}\|_{L^6_{t,x}}\\ &\lesssim_{\delta} M_1\| u\|_{X^{\ga,\f{1}{2}+\de}}  \| v\|_{X^{0,\f{1}{2}}} \| w\|_{X^{0,\f{1}{2}}}
\end{align*}
where $\widetilde{v^{-\delta}} (\tau,\xi) := \langle \xi \rangle^{-\delta} \lan{\tau-\xi^3}^{-\de/2} |\widetilde{v}|(\tau,\xi)$ and 
\begin{align*}
M_1 &:= \sup_{\tiny \begin{array}{c}(\xi_1+\xi_2)(\xi_2+\xi_3)(\xi_3+\xi_1)\neq 0\\ \xi_1+\xi_2+\xi_3=\xi,\quad \xi_j \neq 0 \end{array}} \f{\langle\xi_2\rangle^{s+\delta} \langle \xi_3\rangle^{s+\delta}\langle \xi\rangle^{\ga-s}}{|\xi_3|\lan{\xi_1}^{\ga-s-\de} L_{\max}^{\frac{1}{2}-2\delta}} \\
	&\lesssim \sup_{\tiny \begin{array}{c}(\xi_1+\xi_2)(\xi_2+\xi_3)(\xi_3+\xi_1)\neq 0\\ \xi_1+\xi_2+\xi_3=\xi \end{array}} \frac{L_{\max}^{-\delta} \langle\xi_2\rangle^{s+\delta} \langle \xi\rangle^{\ga-s}}{\lan{\xi_1}^{\ga-s-\de}\lan{\xi_3}^{1-s-\delta} (|\xi_1+ \xi_2||\xi_2+\xi_3||\xi_3+\xi_1|)^{\frac{1}{2}-3\delta} }.
\end{align*}
It now suffices to show that there exists an absolute constant $C$ which bounds
\begin{equation}\label{mpdef}
\frac{\langle\xi_2\rangle^{s+\delta} \langle \xi\rangle^{\ga-s}}{\lan{\xi_1}^{\ga-s-\de}\lan{\xi_3}^{1-s-\delta} (|\xi_1+ \xi_2||\xi_2+\xi_3||\xi_3+\xi_1|)^{\frac{1}{2}-3\delta} }
\end{equation}
for $\xi, \xi_1,\xi_2,\xi_3\in \mathbf{z}$ satisfying $(\xi_1+\xi_2)(\xi_2+\xi_3)(\xi_3+\xi_1)\neq 0$ and $\xi_1+\xi_2+\xi_3=\xi$.

First, consider $|\xi|\sim |\xi_1|\sim |\xi_2| \sim |\xi_3|$, when $|\xi_1+ \xi_2||\xi_2+\xi_3||\xi_3+\xi_1| \gtrsim \xi_{\max}$.  In this case, $\eqref{mpdef} \lesssim \xi_{\max}^{2s+6\de -\f{3}{2}}$ is easily bounded.  Now for the remaining cases, we can assume $|\xi_1+ \xi_2||\xi_2+\xi_3||\xi_3+\xi_1| \gtrsim \xi_{\max}^2$

As expected, the most dangerous scenario is $|\xi|\sim |\xi_2| \gg \max(|\xi_1|,|\xi_3|)$.  In this case, we have $\eqref{mpdef} \lesssim \xi_{\max}^{\ga+7\de-1}$.  \emph{This is the reason for the condition~$\ga<1-10\de$}.  Note that the other cases are strictly nicer and naturally follow from this case.

On the other hand, if $L\ll L_{\max}$, we can reduce it to the first case via an argument similar to one in the proof of Lemma~\ref{le:kspace}.  For example, let $L_2=L_{\max}$.  Using duality, 
\begin{align*}
\n{\mathcal{F}_{\xi}^{-1}[\mathcal{NR}(u,v,w)] }{X^{\ga,-\frac{1}{2}+\delta}} &\lesssim M_1^* L_{\max}^{\f{1}{2}-\de} \n{\frac{\left[\lan{\cdot}^{\ga-\de}|\widetilde{u}|\right]*\left[L_2^{-\f{\de}{2}} |\widetilde{v}|\right] *  \widetilde{w^{\de}}](\tau,\xi)}{\lan{\xi}^{\de}\langle \tau - \xi^3\rangle^{\frac{1}{2}-\delta}}}{L^2_{\tau} l^2_{\xi}}\\
	&\hspace{-120pt}\lesssim M_1^* L_{2}^{\f{1-3\de}{2}}  \sup_{\|\wt{z}\|_{L^2_{\tau}l^2_{\xi}}=1} \left| \int_{\tau_1+\tau_2+\tau_3=\tau} \sum_{\xi_1+\xi_2+\xi_3=\xi} \left[\lan{\xi_1}^{\ga-\de}|\widetilde{u}|\right] |\widetilde{v}|\,\widetilde{w^{\de}}\left[ \f{\lan{\xi}^{-\de}\wt{z}(\tau,\xi)}{\langle \tau-\xi^3 \rangle^{\frac{1}{2}-\de}}\right] d\sigma\right| \\
	&\hspace{-120pt}\lesssim M_1^* L_{2}^{\f{1}{2}}\sup_{\|z\|_{L^2_{\tau}l^2_{\xi}}=1} \n{\left(|\widetilde{v}|\right) \cdot\left(\left[\lan{\cdot}^{\ga-\de} |\widetilde{u}|\right]*\widetilde{w^{\de}}*\left[L^{-\f{1}{2}}\wt{z^{\de}}\right]\right)}{L^1_{\tau}l^1_{\xi}}\\
	&\hspace{-120pt}\lesssim M_1^* \|v\|_{X^{0,\frac{1}{2}}} \sup_{\|z\|_{L^2_{\tau}l^2_{\xi}}=1} \n{\left[\lan{\na}^{\ga-\de} \cf^{-1}_{\tau,\xi}[|\wt{u}|] \right]\,w^{\de}\,\left[L^{-\f{1}{2}}z^{\de}\right] }{L^2_{t,x}}\\
	&\hspace{-120pt}\lesssim_{\delta} M_1^* \|u\|_{X^{0,\frac{1}{2}+\delta}}\|v\|_{X^{0,\frac{1}{2}}}\|w\|_{X^{0,\frac{1}{2}}}
\end{align*}
where we have used H\"older and \eqref{xsb2} for the last inequality, and $M_1^*$ is exactly $ \lan{\xi}^{\de} \lan{\xi_2}^{-\de} M_1$.  Such transfer only results in a harmless exchange of $\lan{\xi}^{\de}$ for $\lan{\xi_2}^{\de}$.  This proves \eqref{eq:wff}.

The proof of \eqref{eq:ffw} is similar, so we will prove this first. 

\textbf{Proof of \eqref{eq:ffw}:} The proof is essentially identical to that of \eqref{eq:wff}.  In place of $M_1$, we have
\begin{align*}
M_2 &:= \sup_{\tiny \begin{array}{c}(\xi_1+\xi_2)(\xi_2+\xi_3)(\xi_3+\xi_1)\neq 0\\ \xi_1+\xi_2+\xi_3=\xi,\quad \xi_j \neq 0 \end{array}} \f{\lan{\xi_1}^{s+\de}\langle\xi_2\rangle^{s+\delta} \langle \xi_3\rangle^{s+\delta}\langle \xi_3\rangle^{\ga-s}}{|\xi_3|\lan{\xi_3}^{\ga-\de} L_{\max}^{\frac{1}{2}-2\delta}} \\
	&\lesssim \sup_{\tiny \begin{array}{c}(\xi_1+\xi_2)(\xi_2+\xi_3)(\xi_3+\xi_1)\neq 0\\ \xi_1+\xi_2+\xi_3=\xi \end{array}} \frac{L_{\max}^{-\delta} \lan{\xi_1}^{s+\de}\langle\xi_2\rangle^{s+\delta} \langle \xi\rangle^{\ga-s}}{\lan{\xi_3}^{1+\ga-s-2\delta} (|\xi_1+ \xi_2||\xi_2+\xi_3||\xi_3+\xi_1|)^{\frac{1}{2}-3\delta} }.
\end{align*}

We omit most cases since they follow the same estimate of $M_1$.  The only case which differs from the estimate of $M_1$ is the case when $\xi_3 \ll \xi_{\max}$.

First, we consider $\xi_1\sim \xi_2\sim  \xi \gg \xi_3$.  Note that in this case, 
\begin{equation}\label{eq:save}
|\xi_1 + \xi_2| |\xi_2 +\xi_3| |\xi_3+\xi_1| = |\xi-\xi_3| |\xi_2 +\xi_3| |\xi_3+\xi_1|
\end{equation}
which has size $\xi_{\max}^3$ under our assumptions.  Thus, $M_2 \lesssim L_{\max}^{-\de} \xi_{\max}^{\ga + s-\f{3}{2} + 11\de}$.

We also should consider $\xi_1 \sim \xi_2 \gg \max(|\xi|,|\xi_3|)$ and $\xi_1 \sim \xi \gg \max(|\xi_2|,|\xi_3|)$.  One needs to be careful here since we cannot assume that small indices are of the similar size.

When $\xi_1 \sim \xi_2 \gg \max(|\xi|,|\xi_3|)$, we use \eqref{eq:save}.  Noting that $\lan{\xi_3} |\xi-\xi_3|\gtrsim \xi$, we can write $\lan{\xi_3}|\xi-\xi_3| |\xi_2 +\xi_3| |\xi_3+\xi_1| \gtrsim \xi_{\max}^2 |\xi|$.  Then, 
\[
M_2 \lesssim L^{-\de}_{\max} \xi_{\max}^{2s+8\de -1} |\xi|^{\ga-s +3\de - \f{1}{2}} \leq L^{-\de}_{\max}|\xi|^{s+\ga+11\de -\f{3}{2}}.
\]

Finally, when $\xi_1 \sim \xi \gg \max(|\xi_2|,|\xi_3|)$, similar argument gives 
\[
M_2 \lesssim L^{-\de}_{\max} \xi_{\max}^{\ga + 7\de - 1} |\xi_2|^{s+4\de -\f{1}{2}}.
\]
Thus, with $\ga <1-10\de$, the desired estimate holds.  This exhaust the cases when $\xi_3 \ll \xi_{\max}$.  For the other cases, refer to the proof of \eqref{eq:wff}.

\textbf{Proof of \eqref{eq:hff}:} The idea of this estimate is very similar to that of \eqref{eq:wff}.  The new ingredient here is that one can choose to switch out modulation frequency for spatial derivative in high frequency realm.  

Begin by writing $\n{\mathcal{F}_{\xi}^{-1}[\mathcal{NR}(u,v,w)] }{X_T^{\ga,-\frac{1}{2}+\delta}}= $
\[
\n{ \sum_{\tiny \begin{array}{c}(\xi_1+\xi_2)(\xi_2+\xi_3)(\xi_3+\xi_1)\neq 0\\ \xi_1+\xi_2+\xi_3=\xi, \quad \xi_j \neq 0 \end{array} } \frac{\langle \xi_1\rangle^s \langle\xi_2\rangle^s \langle \xi_3\rangle^{s}\langle \xi\rangle^{\ga-s}}{i\xi_3\langle\tau-\xi^3 \rangle^{\frac{1}{2}-\delta}} [\widetilde{u}(\xi_1)*_{\tau} \widetilde{v}(\xi_2)*_{\tau} \widetilde{w}(\xi_3)](\tau)}{L^2_{\tau} l^2_{\xi}(\mathbf{Z}\setminus \{0\})}.
\]
We can split the summand into two sets: 1) $L_{\max} \gtrsim \xi_{\max}^2 \min(|\xi_1 + \xi_2|, |\xi_2+\xi_3|,|\xi_3+\xi_1|)$; 2) otherwise.  In the first case, gain through the modulation frequency will be enough, so we will use the $X^{0,\f{1}{2}}$~norm.  In the second case, note that $\xi\sim \xi_1 \sim \xi_2 \sim \xi_3$ must hold, since otherwise \eqref{eq:xisquaregain} forces the first condition to hold.  Therefore, $\xi_1$ is high frequency in this case. So we will use the $X^{1, 0}$~norm.

First, assume $L_{\max} \gtrsim \xi_{\max}^2 \min(|\xi_1 + \xi_2|, |\xi_2+\xi_3|,|\xi_3+\xi_1|)$.  In this case, the computation is almost identical to the proof of \eqref{eq:wff} and it suffices to estimate
\begin{align*}
M_3 &:= \sup_{\tiny \begin{array}{c}(\xi_1+\xi_2)(\xi_2+\xi_3)(\xi_3+\xi_1)\neq 0\\ \xi_1+\xi_2+\xi_3=\xi,\quad \xi_j \neq 0 \end{array}} \f{\lan{\xi_1}^{s+\de} \langle\xi_2\rangle^{s+\delta} \langle \xi_3\rangle^{s+\delta}\langle \xi\rangle^{\ga-s}}{|\xi_3| L_{\max}^{\frac{1}{2}-2\delta}} \\
	&\lesssim \sup_{\tiny \begin{array}{c}(\xi_1+\xi_2)(\xi_2+\xi_3)(\xi_3+\xi_1)\neq 0\\ \xi_1+\xi_2+\xi_3=\xi \end{array}} \frac{L_{\max}^{-\delta} \lan{\xi_1}^{s+\de}\langle\xi_2\rangle^{s+\delta} \langle \xi\rangle^{\ga-s}}{\lan{\xi_3}^{1-s-\delta} \xi_{\max}^{1-6\de} \min(|\xi_1+\xi_2|,|\xi_2 + \xi_3|, |\xi_3 + \xi_1|)^{\f{1}{2}-3\de} }
\end{align*}

If $\xi_3 \sim \xi_{\max}$, then $M_3 \lesssim L_{\max}^{-\de} \xi_{\max}^{\ga+2s -2 +9\de}$ is bounded.  Otherwise $\xi_3 \ll \xi_{\max}$, and the desired estimate follows from the arguments in the proof of \eqref{eq:ffw}.

Next, we can assume $L_{\max} \ll \xi_{\max}^2 \min(|\xi_1 + \xi_2|, |\xi_2+\xi_3|,|\xi_3+\xi_1|)$.  As stated above, this forces all frequencies to be at the same level.  Following previous computations,
\begin{align*}
\n{\mathcal{F}_{\xi}^{-1}[\mathcal{NR}(u,v,w)] }{X^{\ga,-\frac{1}{2}+\delta}} &\lesssim M_3^*  \n{\frac{\left[\lan{\cdot}|\widetilde{u}|\right]*\widetilde{v^{\de}} *  \widetilde{w^{\de}}](\tau,\xi)}{\lan{\xi}^{\de} L^{\frac{1}{2}+\de}}}{L^2_{\tau} l^2_{\xi}}\\
	&\lesssim M_3^* \|u\|_{X^{1,0}} \sup_{\|z\|_{L^2_{\tau}l^2_{\xi}}=1} \n{v^{\de}\,w^{\de}\,\left[L^{-\f{1+\de}{2}}z^{\de}\right] }{L^2_{t,x}}\\
	&\lesssim_{\delta} M_3^* \|u\|_{X^{1,0}}\|v\|_{X^{0,\frac{1}{2}}}\|w\|_{X^{0,\frac{1}{2}}}
\end{align*}
where 
\begin{align*}
M_3^* &:= \sup_{\tiny \begin{array}{c}(\xi_1+\xi_2)(\xi_2+\xi_3)(\xi_3+\xi_1)\neq 0\\ \xi_1+\xi_2+\xi_3=\xi,\quad \xi_j \neq 0 \end{array}} \f{L_{\max}^{3\de} \langle\xi_2\rangle^{s+\delta} \langle \xi_3\rangle^{s+\delta}\langle \xi\rangle^{\ga-s+\de}}{\lan{\xi_1}^{1-s}|\xi_3| }\\
&:= \sup_{\tiny \begin{array}{c}(\xi_1+\xi_2)(\xi_2+\xi_3)(\xi_3+\xi_1)\neq 0\\ \xi_1+\xi_2+\xi_3=\xi,\quad \xi_j \neq 0 \end{array}} \f{L_{\max}^{-\de}\xi_{\max}^{12\de} \langle\xi_2\rangle^{s+\delta} \langle \xi\rangle^{\ga-s+\de}}{\lan{\xi_1}^{1-s}\lan{\xi_3}^{1-s-\de} }.
\end{align*}
Since all $\xi\sim \xi_1\sim \xi_2 \sim \xi_3$, we have $M_3^* \lesssim L_{\max}^{-\de}\xi_{\max}^{\ga +2s +15\de -2}$.  Thus we have \eqref{eq:hff}.

\eqref{eq:ffh} can be proved exactly the same way, thus the proof is omitted.
\end{proof}

The next lemma deals with the \emph{resonant} term~$\mathcal{R}$ in \eqref{eqwper}.  To reduce the number of cases, we ignore the complex conjugation.  In particular, this means that $\mathcal{R}(\cdot, \cdot,\cdot)$ can be considered to be symmetric in all three variables.  This does not cause any problem in the proof, since we do not take advantage of cancellations here.  

\begin{lemma}\label{res}
Given $0\leq s<1/2$, $0<\ga\leq 1$ and $0<\de<1/2$, the following estimate holds for arbitrary functions $u,v\in L^{\infty}_t([0,T];L^2_x(\bt))$ and $w\in L^2 H_x^{\ga}$.
\[
\n{\cf^{-1}\left[\mathcal{R}(u,v,w)\right]}{X_T^{\ga, -\f{1}{2}+\de}} \lesssim_{\de} T^{\f{1}{2}-\de}\|u\|_{L^{\infty}_t([0,T];L^2_x(\bt))} \|v\|_{L^{\infty}_t([0,T];L^2_x(\bt))} \|w\|_{L^{\infty}_t([0,T];H^{\ga}_x(\bt))}.
\]
\end{lemma}

\textbf{Remark:} The estimate above suffices to control all the resonance because 
\[
R[f]\in L^{\infty}_T L^2_x;\quad h,k \in L^{\infty}_T L^2_x \cap L^{2}_T H^{\ga}_x; \quad w\in X_T^{\ga, \f{1}{2}+\de}\subset L^{2}_T H^{\ga}_x \cap C^0_t H^{\ga}_x.
\]

\begin{proof}
In the resonant part, we do not benefit from modulation frequencies, so we can enjoy a large power of $T$.  We have
\begin{align*}
\n{\cf^{-1}\left[\mathcal{R}(u,v,w)\right]}{X_T^{\ga, -\f{1}{2}+\de}} &\lesssim_{\de} T^{\f{1}{2}-\de} \n{\frac{\langle \xi \rangle^{2s+\ga}}{\xi} |\widehat{u}| |\wt{v}| |\wh{w}|}{L^2_T l^2_{\xi}}\\ 
	&\lesssim  T^{\f{1}{2}-\de} \n{|\widehat{u}| |\wh{v}| \left[\langle \xi \rangle^{\ga} |\wh{w}|\right]}{L^2_T l^2_{\xi}}\\ 
	&\lesssim  T^{\f{1}{2}-\de}\n{\widehat{u}}{L^{\infty}_T l^{\infty}_{\xi}}\n{\widehat{v}}{L^{\infty}_T l^{\infty}_{\xi}} \n{\left[\langle \xi \rangle^{\ga} |\wh{w}|\right]}{L^2_T l^2_{\xi}}\\ 
	&\lesssim  T^{\f{1}{2}-\de}\n{u}{L^{\infty}_T L^2_{x}} \n{v}{L^{\infty}_T L^{2}_{x}} \n{w}{L^2_T H^{\ga}_x}. 
\end{align*}
\end{proof}

Finally, it remains to estimate the quinti-linear term~$\mathcal{Q}$ in \eqref{eqwper}.  As observed before, this will be essentially a trilinear estimate, so we will use the $L^6$~embedding.

\begin{lemma} \label{le:quintic}
Given $0\leq s <1/2$, $0<10\de<1-2s$ and $0<\ga\leq 1$, the following estimates hold for $R[f]$ and any $u\in X^{0,\f{1}{2}+\de}$.
\begin{align}
\n{J(\cf^{-1}_{\xi}[\mathcal{R}(R[f],R[f],R[f])],u,u)}{X_T^{\ga, -\f{1}{2}+\de}} &\lesssim_{\de} T^{\de} \|f\|_{L^2}^2 \|R[f]\|_{X^{0,\f{1}{2}+\de}} \|u\|^2_{X^{0,\f{1}{2}+\de}}\label{eq:j1}\\
\n{J(u,u,\cf^{-1}_{\xi}[\mathcal{R}(R[f],R[f],R[f])])}{X_T^{\ga, -\f{1}{2}+\de}} &\lesssim_{\de} T^{\de} \|f\|_{L^2}^2\|R[f]\|_{X^{0,\f{1}{2}+\de}} \|u\|^2_{X^{0,\f{1}{2}+\de}}\label{eq:j2}
\end{align}
\end{lemma}

\begin{proof}
First, we consider \eqref{eq:j1}.  As in Lemma~\ref{nonres}, we decompose the modulation frequency such that $\lan{\tau_j - \xi_j^3} \sim L_j$ and $\lan{\tau-\xi^3} \sim L$ for dyadic indices~$L, L_j \geq 1$.  We will show that each modulational component is bounded by $C\|f\|_{L^2}^2 L_{\max}^{-\de}$ where $C\leq C(\de)$. 

First we consider the case when $L\sim L_{\max}$. Using Plancherel, H\"older and $L^6$ embedding~\eqref{xsb2}, we show $\n{J(\cf^{-1}_{\xi}[\mathcal{R}(R[f],R[f],R[f])],u,u)}{X_T^{\ga, -\f{1}{2}+\de}}  = $
\begin{align*}
&\n{\sum_{\tiny \begin{array}{c} (\xi_1 + \xi_2)(\xi_2+ \xi_3)(\xi_3+\xi_1)\neq 0\\ \xi_1 + \xi_2 + \xi_3 = \xi, \qquad \xi_j \neq 0\end{array}} \f{\lan{\xi_1}^{3s} \lan{\xi_2}^s \lan{\xi_3}^s |\wh{f}(\xi_1)|^2\left[\wh{R[f]}(\xi_1) *_{\tau} \wh{u}(\xi_2) *_{\tau}\wh{u}(\xi_3)\right]}{-\lan{\tau- \xi^3}^{\f{1}{2}-\de}\xi_1 \xi_3  \lan{\xi}^s (\xi_1 + \xi_2)(\xi_2+ \xi_3)(\xi_3+\xi_1)} }{L^2_{\tau} l^2_{\xi}}\\
	&\lesssim_{\delta} M \n{ \wt{R[f]^{-\de}}* ( \widetilde{v^{-\delta}} * \widetilde{w^{-\de}})}{L^2_{\tau}l^2_{\xi}}  \lesssim_{\delta} M\| R[f]\|_{X^{0,\f{1}{2}+\de}}  \| v\|_{X^{0,\f{1}{2}+\de}} \| w\|_{X^{0,\f{1}{2}+\de}}
\end{align*}
where $v^{-\de}$ is defined by $\wt{v^{-\de}}(\xi) =  \lan{\xi}^{-\de}|\wt{v}|(\tau,\xi)$ and
\begin{align*}
M &:= \sup_{\tiny \begin{array}{c}(\xi_1+\xi_2)(\xi_2+\xi_3)(\xi_3+\xi_1)\neq 0\\ \xi_1+\xi_2+\xi_3=\xi,\quad \xi_j \neq 0 \end{array}} \frac{\lan{\xi_1}^{3s+\de} \lan{\xi_2}^{s+\de}\lan{\xi_3}^{s+\de}\lan{\xi}^{\ga-s} \left|\wh{f}(\xi_1)\right|^2 }{L_{\max}^{\f{1}{2}-\de} |\xi_1| |\xi_3|  |(\xi_1 + \xi_2)(\xi_2+ \xi_3)(\xi_3+\xi_1)|}\\
	 &\lesssim \sup_{\tiny \begin{array}{c}(\xi_1+\xi_2)(\xi_2+\xi_3)(\xi_3+\xi_1)\neq 0\\ \xi_1+\xi_2+\xi_3=\xi,\quad \xi_j \neq 0 \end{array}} \frac{L_{\max}^{-\de}\|f\|_{L^2}^2\lan{\xi_1}^{3s-1+\de} \lan{\xi_2}^{s+\de}\lan{\xi}^{\ga-s}}{\lan{\xi_3}^{1-s-\de}  |(\xi_1 + \xi_2)(\xi_2+ \xi_3)(\xi_3+\xi_1)|^{\f{3}{2}-2\de}}.
\end{align*}	 

So it suffices to find an absolute constant $C$ to bound
\begin{equation}\label{eq:lastbound}
\frac{\lan{\xi_1}^{3s-1+\de} \lan{\xi_2}^{s+\de}\lan{\xi}^{\ga-s}}{\lan{\xi_3}^{1-s-\de}  |(\xi_1 + \xi_2)(\xi_2+ \xi_3)(\xi_3+\xi_1)|^{\f{3}{2}-2\de}}
\end{equation}
for $\xi, \xi_1, \xi_2,\xi_3\in bz\setminus\{0\}$ satisfying $(\xi_1+\xi_2)(\xi_2+\xi_3)(\xi_3+\xi_1)\neq 0$ and $\xi_1+\xi_2+\xi_3=\xi$.

Recalling $(\xi_1 + \xi_2)(\xi_2+ \xi_3)(\xi_3+\xi_1)\gtrsim \xi_{\max}$, we have $\eqref{eq:lastbound} \lesssim \xi_{\max}^{\ga +s +3\de -3/2}$ which is bounded.

We remark that when $L\ll L_{\max}$, then we can transfer $L_{\max}$ from elsewhere as done in Lemma~\ref{nonres} possibly at the cost of $\xi_{\max}^{\de}$.

To show \eqref{eq:j2}, we follow the same computations and reduce to finding the bound for 
\begin{equation}\label{eq:lastbound2}
\frac{\lan{\xi_1}^{s+\de} \lan{\xi_2}^{s+\de}\lan{\xi}^{\ga-s}}{\lan{\xi_3}^{2-3s-\de}  |(\xi_1 + \xi_2)(\xi_2+ \xi_3)(\xi_3+\xi_1)|^{\f{3}{2}-2\de}}.
\end{equation}

Using $|(\xi_1 + \xi_2)(\xi_2+ \xi_3)(\xi_3+\xi_1)|\gtrsim \xi_{\max}$, we obtain $\eqref{eq:lastbound2} \lesssim \xi_{\max}^{\ga+s+4\de -\f{3}{2}}$. This proves \eqref{eq:j2}.
\end{proof}

\subsection{Local theory}
\label{sec:local}
We now have obtained all necessary estimates to perform contraction on $Y^{\ga,\frac{1}{2}+\delta}$ for $0<10\de<1-2s$ and $\ga\leq 1-10\de$.  In this section, we assume $T>0$ is small.

\begin{proposition}\label{pro:local}
The PBIVP \eqref{eqwper} is locally well-posed in $H^{\ga}$ where $\ga$ satisfies  $\ga\leq 1-10\de$ for some $\de>0$ such that $0<10\de<1-2s$.  Furthermore, there exists $T = O \left( \|f\|_{L^2}^{-10/\de}\right)$, such that the solution $w(t)$ for $t\in [0,T]$ solves \eqref{eqwper} and satisfies
\[
\|w\|_{X^{\ga,\f{1}{2}+\de}} \sim \|w(0)\|_{H^{\ga}}.
\]
\end{proposition}

\begin{proof}
We write \eqref{eqwper} via the Duhamel's formula,
\begin{equation}\label{duhamel}
w(t) = \eta(t) e^{-t\p_x^3} w_0 + \eta(t/T)\int_0^t e^{-(t-s)\p_x^3} \left[\mathcal{NR} + \mathcal{R} + \mathcal{Q}\right]\, ds
\end{equation}
where $w_0 = w(0)$ given in \eqref{eqwper}.  We define the map~$\La_T: Y^{0,\f{1}{2}+\de} \to Y^{0,\f{1}{2}+\de}$ so that $\La_T(w)$ is the RHS of \eqref{duhamel} for arbitrary function $w\in Y^{0,\f{1}{2}+\de}$.

Let $\mathcal{B}$ be a ball in $Y^{\ga,\frac{1}{2}+\delta}$ centered at $-\eta(t) e^{-t\p_x^3} w_0$ with small radius.  We now show that for $T$ sufficiently small, $\La_T$ is a contraction on $\mathcal{B}$.  By Propositions~\ref{xsb} and \ref{timeloc}, we have
\[
\n{\Lambda_T w - e^{-t\p_x^3} w_0}{Y^{\ga,\frac{1}{2}+\delta}} \lesssim_{\eta} \|\mathcal{NR}\|_{Y_T^{\ga,-\frac{1}{2}+\delta}}+ \|\mathcal{R}\|_{Y_T^{\ga,-\frac{1}{2}+\delta}} + \|\mathcal{R}\|_{Y_T^{\ga,-\frac{1}{2}+\delta}}.
\]

Applying Lemmas~\ref{nonres}, \ref{res} and \ref{le:quintic}, and noting the nonlinearies \eqref{eq:nr1}-\eqref{eq:q2}, we can estimate above by
\begin{align}
C_{\de} T^{\de}&\left[\left(\|w\|_{X^{\ga,\f{1}{2}+\de}} + \|h+k\|_{X_1^{1,0}\cap X_1^{0,\f{1}{2}}}\right)\n{R[f]+h+k+w}{X_1^{0,\f{1}{2}}}^2 \right.\notag\\
&\left.+\|w+h+k\|_{L_t^{\infty}([0,T]; H^{\ga}_x)}\n{R[f]+h+k+w|}{L^{\infty}_t([0,T];L^2_x)}+ \|f\|^2_{L^2} \n{R[f]}{X_1^{0,\f{1}{2}+\de}}^3\right].\label{finalcont}
\end{align}

To simplify the estimates, assume $\|f\|_L^2\gg 1$ so that we may drop all lower powers.  By Lemmas~\ref{tpmap1} and \ref{le:jmap} and since $w\in \mathcal{B}$, we have
\[
\|w\|_{X^{\ga,\f{1}{2}+\de}} \sim \|w_0\|_{H^{\ga}} \lesssim \|T(f,f)\|_{H^{\ga}} + \|J(f,f,f)\|_{H^{\ga}} \lesssim \|f\|_{L^2}^2 + \|f\|_{L^2}^3\lesssim \|f\|_{L^2}^3.
\]
We use Lemma~\ref{rffree} for $R[f]$ and estimates\eqref{eq:tspace}, \eqref{eq:jspace} for $h,k$ respectively to show
\begin{align*}
\eqref{finalcont} &\leq C_{\de} T^{\de} \left(\|v\|_{X^{0,\f{1}{2}}\cap \wt{l^2_{\xi} L^1_{\tau}}}^2 + \|f\|_{L^2}^5\right)\left(\|f\|_{L^2}^5 +  \|v\|_{X^{0,\f{1}{2}} \cap \wt{l^2_{\xi} L^1_{\tau}}}^2\right)\leq C_{\de} T^{\de} \|f\|^{10}
\end{align*}
where we used Proposition~\ref{pro:timebound} to obtain $\|v\|_{X^{0,\f{1}{2}} \cap \wt{l^2_{\xi} L^1_{\tau}}} \lesssim \|f\|_{L^2}$ for short time~$T$.  Thus we need $T^{\de} \ll \|f\|^{-10}$.

It is easy to see from above computations that $T\ll \|f\|^{-10/\de}$ suffices for contraction also.  This proves the desired local well-posedness.
\end{proof}

The following is the main corollary from the local theory, which iterated in the next section to extend the result globally in time.

\begin{corollary}\label{cor:localsmoothing}
Given the initial data~$f$ in \eqref{periodic} and $\de,\ga,T >0$ as in Proposition~\ref{pro:local}, the local-in-time solution satisfies
\[
\sup_{t\in [0,T]} \|v(t)-R[f](t)\|_{H^{\ga}_x} \leq C_{\de} (1+ \|f\|_{L^2})^5
\]
\end{corollary}

\begin{proof}
Since $v= R[f] + h + k + w$, it suffices to estimate
\[
\|h+k+w\|_{L^{\infty}_t([0,T]; H^{\ga}_x)} \lesssim_{\de}  \|h\|_{L^{\infty}_t([0,T]); H^1_x} + \|k\|_{X^{0,\f{1+\de}{2}}_T} + \|w\|_{X^{\ga, \f{1}{2}+\de}}.
\]

The first two terms are estimated through Lemma~\ref{tpmap1} and \eqref{eq:jspace} respectively, and the last term is given in Proposition~\ref{pro:local}.
\end{proof}

\subsection{Conclusion of the proof of Theorem~\ref{thm2}}\label{global}
With the local theory from Section~\ref{sec:local}, we iterate to obtain the corresponding global theory.  This is not very easy in this setting because $R[f]$ is not linear.  However, we can overcome this and still iterate as desired.  In this section, we assume $T>0$ is arbitrarily large.

Let $\ga, \de>0$ be as in Proposition~\ref{pro:local}.  Given $T>0$ large, Proposition~\ref{pro:timebound} gives that $\sup_{t\in [0,T]} \|v(t)\|_{L^2} \lesssim \lan{T}^s \|f\|_{L^2}$.  Then we select the time increment $\ve = O_{\|f\|_{L^2}}(\lan{T}^{-10s/\de})$ fixed in with the local theory holds as stated in Section~\ref{sec:local}.  The number of iterations~$M$ is of order $T/\ve$, so we let $M = O_{\|f\|_{L^2}}(T^{1+10s/\de})$.

Our iterating mechanism is as follows.  Uniformly partition the time interval~$[0,T]$ into $M$ pieces and let $\ve = T/M$.   Denote $v_j := v(j\ve)$. By triangular inequality,
\begin{equation}\label{eq:it}
\|v(T) - R[f](T)\|_{H^{\ga}} \leq \sum_{j=1}^{M} \n{R[v_j]((M-j)\ve)- R[v_{j-1}]((M-j+1)\ve)}{H^{\ga}}.
\end{equation}
The summand on the RHS of \eqref{eq:it} can be expressed as
\begin{equation}\label{eq:itsum}
\n{\lan{\xi}^{\ga}\left(\wh{v_j}e^{-2i\f{\lan{\xi}^{2s}}{\xi}|\wh{v_j}|^2 (M-j)\ve} - \wh{v_{j-1}} e^{-2i\f{\lan{\xi}^{2s}}{\xi}|\wh{v_{j-1}}|^2(M-j+1)\ve}\right)}{l^2_{\xi}}
\end{equation}

Adding and subtracting
\[
\lan{\xi}^{\ga} \wh{v_j} \exp\left(-2i\f{\lan{\xi}^{2s}}{\xi} |\wh{v_{j-1}}(\xi)|^2 (M-j)\ve\right),
\]
we note that \eqref{eq:itsum} splits into two piecies:
\begin{align}
&\n{\lan{\xi}^{\ga} \wh{v_j} \left(\exp \left[-2i \f{\lan{\xi}^{2s}}{\xi}\left( |\wh{v_j}|^2 - |\wh{v_{j-1}}|^2\right)(M-j)\ve\right] -1 \right)}{l^2_{\xi}}\label{eq:it1}\\
&+ \n{\lan{\xi}^{\ga}\left(\wh{v_j} - \wh{v_{j-1}} e^{-2i \f{\lan{\xi}^{2s}}{\xi} |\wh{v_{j-1}}|^2\ve}\right)}{l^2_{\xi}}.\label{eq:it2}
\end{align}
The second piece~\eqref{eq:it2} is equivalent to $\|v_{j} - R[v_{j-1}](\ve)\|_{H^{\ga}}$. Thus, using Corollary~\ref{cor:localsmoothing}, Proposition~\ref{pro:timebound} and assuming $T, \|f\|_{L^2}\gg 1$
\[
\eqref{eq:it2} \leq C_{\de} (1+\|v_{j-1}\|_{L^2})^5 \leq C_{\de} T^{5s} \|f\|_{L^2}^5.
\]

To estimate \eqref{eq:it1}, we use the mean-value theorem,
\begin{align*}
\eqref{eq:it1} &\lesssim T \n{\lan{\xi}^{\ga} \wh{v_j} \f{\lan{\xi}^{2s}}{\xi} \left(\left|\wh{v_j}(\xi)\right|^2 - \left|\wh{v_{j-1}}(\xi)\right|^2\right) }{l^2_{\xi}}\\
	&\lesssim T \n{\lan{\xi}^{\ga}\left(\left|\wh{v_j}\right| - \left|\wh{v_{j-1}}\right|\right)}{l^{2}_{\xi}} \left(\n{\wh{v_j}}{l^{\infty}_{\xi}} + \n{\wh{v_{j-1}}}{l^{\infty}_{\xi}}  \right) \n{\wh{v_j}}{l^{\infty}_{\xi}}
\end{align*}
Apart from the first term, we can estimate $\n{\wh{v_j}}{l^{\infty}_{\xi}} \leq \n{v_j}{L^2} \lesssim T^{s}\|f\|_{L^2}$.  To estimate the first term, note that we can replace 
\[
\left|\wh{v_{j-1}}(\xi)\right| = \left| \wh{v_{j-1}}(\xi) \exp \left(-2i \f{\lan{\xi}^{2s}}{\xi} |\wh{v_{j-1}}(\xi)|^2 \ve\right)\right|.
\]
Then using triangular inequality, this term is bounded by
\[
\n{\lan{\xi}^{\ga}\left( \wh{v_j}(\xi) - \wh{v_{j-1}}(\xi) e^{-2i \f{\lan{\xi}^{2s}}{\xi} |\wh{v_{j-1}}(\xi)|^2 \ve} \right)}{l^2}
\]
which is equivalent to $\n{v_j- R[v_{j-1}](\ve)}{H^{\ga}}$ which is bounded by $C_{\de}T^{5s} \|f\|^5_{L^2}$ from Corollary~\ref{cor:localsmoothing}.  Thus 
\[
\eqref{eq:it1} \leq C_{\de} T^{1+7s} \|f\|_{L^2}^5.
\]
Summing in $j=1,\cdots, M$, we have
$\|v(T) - R[f](T)\|_{H^{\ga}} \leq C_{\de} M T^{1+7s} \|f\|_{L^2}^5$.  Considering $M = O_{\|f\|_{L^2}} (T^{1+10s/\de})$, 
\[
\|v(T) - R[f](T)\|_{H^{\ga}} \leq C_{\de, \|f\|_{L^2}} T^{2+7s+ 10s/\de}. 
\]
Thus we obtain Theorem~\ref{thm2} with $\al(\de) = 2+7s + 10s/\de$.

To prove the Lipschitz property~\eqref{eq:thmlip}, let $f^1, f^2 \in L^2$ and $\|f^1 - f^2\|_{H^{\ga}}<\infty$ be the initial data of \eqref{periodic}.  Decompose the corresponding solutions $v^j = R[f^j] + h^j + k^j + w^j$ for $j=1,2$.  From Lemma~\ref{lipschitz},
$\|R[f^1]-R[f^2]\|_{C^0_t([0,T];H^{\ga}_x)} \lesssim \|f^1 - f^2\|_{H^{\ga}}$.  For $h^j := T(v^j, v^j)$, we use Lemma~\ref{tpmap1}
\begin{align*}
\|h^1 - h^2\|_{L^{\infty}_t([0,T];H^{\ga}_x)} &\lesssim \n{T(u+v, u-v)}{L^{\infty}_t([0,T];H^{\ga}_x)}\\
	 &\lesssim \|v^1+v^2\|_{L^{\infty}_T L^2} \|v^1-v^2\|_{L^{\infty}_t L^2}\\
	 &\lesssim T^{2s} (\|f^1\|_{L^2}+\|f^2\|_{L^2})\|f^1 - f^2\|_{L^2}.
\end{align*}
where the last statement follows from the local and global well-posedness theory of \cite{I1}.  For $k^j$,
\begin{align*}
\|k^1 - k^1\|_{L^{\infty}_T H^{\ga}} &= \n{J(R[f^1],R[f^1],R[f^1])- J(R[f^2],R[f^2],R[f^2])}{X^{\ga,\f{1+\de}{2}}_T}\\
	&\hspace{-60pt}= \n{J(R[f^1] - R[f^2], R[f^1],R[f^1]) + J(R[f^2], R[f^1]+ R[f^2], R[f^1]- R[f^2])}{X^{\ga,\f{1+\de}{2}}_T}.
\end{align*}
Applying Lemmas~\ref{le:kspace}, \ref{rffree} and \ref{lipschitz}, we obtain the desired estimate.

The Lipschitz continuity of $w^j$ follows somewhat indirectly from the local theory in Proposition~\ref{pro:local}.  Note that $w^j$ can be regarded as a perturbation of the free solution $-e^{-t\p_x^3} w^j(0)$ where $w^j(0) = -T(f^j,f^j) - J(f^j,f^j,f^j)]$ for short time.  Since
\[
\|e^{-t\p_x^3} w^1(0) - e^{-t\p_x^3} w^2(0)\|_{H^{\ga}} \lesssim_{\|f^1\|_{L^2}, \|f^2\|_{L^2}} \|f^1 - f^2\|_{L^2},
\]
 by Lemmas~\ref{tpmap1} and \ref{le:jmap}, we gain local Lipschitz property from Proposition~\ref{pro:local}.  Thus the Lipschitz property of $w^j$ follows from standard iteration. This proves Theorem~\ref{thm2}.

\bibliographystyle{plain}
\bibliography{bib}

\end{document}